\documentclass{amsart}
\usepackage[latin1]{inputenc}
\usepackage{amssymb,amsmath,comment}
\usepackage[nobysame]{amsrefs}
\usepackage[usesnames,svgnames,table]{xcolor}
\usepackage{booktabs}
\usepackage{tikz,pgf}
  \pgfdeclarelayer{background}
  \pgfdeclarelayer{middle}
  \pgfsetlayers{background,middle,main}
\usepackage{enumerate}
\usepackage{multirow}
\usepackage[sc]{mathpazo}
\usepackage[colorlinks=true,
  urlcolor=MidnightBlue,
  citecolor=DarkGreen]{hyperref}
\usepackage{cleveref}
\crefformat{footnote}{#2\footnotemark[#1]#3}
\crefformat{conjecture}{Conjecture~#2#1#3}
\makeatletter
\let\orgdescriptionlabel\descriptionlabel
\renewcommand*{\descriptionlabel}[1]{%
  \let\orglabel\label
  \let\label\@gobble
  \phantomsection
  \edef\@currentlabel{#1}%
  \let\label\orglabel
  \orgdescriptionlabel{(#1)}%
}
\newtheorem{theorem}{Theorem}[section]
\newtheorem{lemma}[theorem]{Lemma}
\newtheorem{proposition}[theorem]{Proposition}
\newtheorem{conjecture}[theorem]{Conjecture}

\theoremstyle{remark}
\newtheorem{definition}[theorem]{Definition}
\newtheorem{remark}[theorem]{Remark}
\newtheorem{example}[theorem]{Example}

\newcommand{\defn}[1]{\emph{\color{DarkGreen} #1}}
\newcommand{\ie}{\text{i.e.}\;}
\DeclareMathOperator{\inv}{\mathrm{Inv}}
\DeclareMathOperator{\des}{\mathrm{Des}}
\DeclareMathOperator{\cov}{\mathrm{Cov}}

\DeclareMathOperator{\Weak}{\mathrm{Weak}}

\DeclareMathOperator{\Align}{\mathrm{Align}}

\DeclareMathOperator{\Cat}{\mathrm{Cat}}
\newcommand{\Tamari}{\mathcal{T}}
\DeclareMathOperator{\NC}{\mathrm{NC}}
\DeclareMathOperator{\NN}{\mathrm{NN}}
\DeclareMathOperator{\B}{\mathrm{B}}
\DeclareMathOperator{\SW}{\mathrm{SW}}
\newcommand{\id}{e}
\newcommand{\wo}{w_{\circ}}
\newcommand{\Pf}{\mathbf{P}}
\newcommand{\Inv}[1]{\textbf{Inv}(\mathbf{#1})}

\newcommand{\ncbody}{
  \def\xx{.25};
  \draw(1*\xx,1) node[circle,fill,scale=.33](p1){};
  \draw(1*\xx,.8) node{$1$};
  \draw(2*\xx,1) node[circle,fill,scale=.33](p2){};
  \draw(2*\xx,.8) node{$2$};
  \draw(3*\xx,1) node[circle,fill,scale=.33](p3){};
  \draw(3*\xx,.8) node{$3$};
  \draw(4*\xx,1) node[circle,fill,scale=.33](p4){};
  \draw(4*\xx,.8) node{$4$};
  \begin{pgfonlayer}{background}
    \filldraw[fill=white,draw=black](.6*\xx,.7) -- (4.4*\xx,.7) -- (4.4*\xx,1.2) -- (.6*\xx,1.2) -- (.6*\xx,.7);
    \fill[gray!50!white](1*\xx,1) circle(.07);
    \fill[gray!50!white](2*\xx,1) circle(.07);
    \fill[gray!50!white](3*\xx,1) circle(.07);
    \fill[gray!50!white](4*\xx,1) circle(.07);
    \fill[gray!50!white](2*\xx,.93) -- (3*\xx,.93) -- (3*\xx,1.07) -- (2*\xx,1.07) -- (2*\xx,.93);
  \end{pgfonlayer}
}
\newcommand{\nnbody}{
  \draw(.2,1) node[circle,fill,scale=.5](p1){};
  \draw(.6,1) node[circle,fill,scale=.5](p2){};
  \draw(.3,1.2) node[circle,fill,scale=.5](p3){};
  \draw(.5,1.2) node[circle,fill,scale=.5](p4){};
  \draw(.4,1.4) node[circle,fill,scale=.5](p5){};
  \draw(p1) -- (p3);
  \draw(p2) -- (p4);
  \draw(p3) -- (p5);
  \draw(p4) -- (p5);
  \begin{pgfonlayer}{background}
    \filldraw[fill=white,draw=black](.1,.9) -- (.7,.9) -- (.7,1.5) -- (.1,1.5) -- (.1,.9);
  \end{pgfonlayer}
}
\author{Henri M{\"u}hle}
\address{ Technische Universit{\"a}t Dresden, Institut f{\"u}r Algebra, Dresden, Germany.}
\email{henri.muehle@tu-dresden.de}
\author{Nathan Williams}
\address{University of Texas at Dallas.}
\email{nathan.f.williams@gmail.com}
\thanks{HM was partially supported by a Public Grant overseen by the French National Research Agency (ANR) as part of the ``Investissements d'Avenir'' Program (Reference: ANR-10-LABX-0098), and by Digiteo project PAAGT (Nr. 2015-3161D). NW was partially supported by a Simons Foundation Collaboration Grant.}
\title{Tamari Lattices for Parabolic Quotients of the Symmetric Group}

\begin{document}

\begin{abstract}
    We generalize the Tamari lattice by extending the notions of $231$-avoiding permutations, noncrossing set partitions, and nonnesting set partitions to parabolic quotients of the symmetric group $\mathfrak{S}_{n}$.  We show bijectively that these three objects are equinumerous.  We show how to extend these constructions to parabolic quotients of any finite Coxeter group.  The main ingredient is a certain aligned condition of inversion sets; a concept which can in fact be generalized to any reduced expression of any element in any (not necessarily finite) Coxeter group.
\end{abstract}
\maketitle

\section{Introduction}
  \label{sec:introduction}

\subsection{Parabolic Tamari Lattices}
The \defn{Tamari lattice} $\Tamari_{n}$ was introduced by D.~Tamari as a partial order encoding the associativity of the Catalan-many binary bracketings of a word of length $n{+}1$~\cite{tamari51monoides}.   The \defn{weak order} $\Weak(\mathfrak{S}_n)$ on the group of permutations $\mathfrak{S}_{n}$ is the oriented Cayley graph of $\mathfrak{S}_n$, using the generating set $S$ of adjacent transpositions.   Rephrasing slightly, A.~Bj{\"o}rner and M.~Wachs realized $\Tamari_{n}$ as a sublattice of $\Weak(\mathfrak{S}_{n})$ by considering the subset $\mathfrak{S}_n(231)\subseteq \mathfrak{S}_n$ of $231$-avoiding permutations, whose inversions sets they characterize as ``compressed''~\cite{bjorner97shellable}*{Section~9}.  N.~Reading extended this result by noting that $\Tamari_n$ was a lattice quotient of $\Weak(\mathfrak{S}_n)$~\cite{reading06cambrian}.

We generalize the Tamari lattice from the symmetric group to its parabolic quotients.  Any $J\subseteq S$ defines the parabolic quotient $\mathfrak{S}_{n}^{J}$, consisting of those permutations $w \in \mathfrak{S}_n$ with descents only at positions not in $J$---that is, if $w(i)>w(i+1)$, then $s_i \not \in J$.  Parabolic quotients have a weak order $\Weak(\mathfrak{S}_{n}^{J})$ inherited from $\Weak(\mathfrak{S}_n)$.    We specify a subset $\mathfrak{S}_{n}^{J}(231) \subseteq \mathfrak{S}_{n}^{J}$ by introducing a generalized notion of $231$-avoidance, dependent on $J$, which can again be seen as a ``compressed'' condition on inversion sets.   

\begin{theorem}\label{thm:parabolic_tamari_lattice}
  Let $n>0$. For $J\subseteq S$, the restriction $\Weak(\mathfrak{S}_n^J)$ to $\mathfrak{S}_n^J(231)$ is a lattice, which we denote $\Tamari_n^J$.  Although $\Tamari_n^J$ is not generally a sublattice of $\Weak(\mathfrak{S}_n^J)$, it is a lattice quotient of $\Weak(\mathfrak{S}_n^J)$.
\end{theorem}

When $J=\emptyset$, we recover the Tamari lattice on $231$-avoiding permutations, and we therefore refer to $\Tamari_{n}^{J}$ as the \defn{parabolic Tamari lattice}.\footnote{This nomenclature is slightly ambiguous, since it could refer to either parabolic quotients or parabolic subgroups---but the Tamari lattice of a parabolic subgroup is a direct product of Tamari lattices, and so deserves no special name.}

\subsection{Parabolic Catalan Objects}
In recent years, many combinatorial families enumerated by the Catalan numbers have been linked to the symmetric group.  Two prototypical examples for this phenomenon are the noncrossing set partitions---which are the elements in an interval in the absolute order for $\mathfrak{S}_n$---and the nonnesting set partitions---which are order ideals in the root poset of $\mathfrak{S}_n$.  We propose new generalizations of noncrossing and nonnesting partitions to $\mathfrak{S}_{n}^{J}$, and denote the resulting sets by $\NC_{n}^{J}$ and $\NN_{n}^{J}$, respectively.  For $J=\emptyset$, we recover the classical noncrossing and nonnesting set partitions, respectively.

The well-known property that the 231-avoiding permutations, noncrossing and nonnesting partitions are equinumerous, survives our generalization to parabolic quotients.

\begin{theorem}\label{thm:parabolic_bijections}
  For $n>0$ and $J\subseteq S$, we have bijections
  \begin{displaymath}
    \mathfrak{S}_{n}^{J}(231) \simeq \NC^J_n \simeq \NN^J_n.
  \end{displaymath}
\end{theorem}

Although we no longer have a nice closed formula in general, the parabolic nonnesting partitions enable us to write down a determinantal formula for the \defn{parabolic Catalan numbers}, see the end of \Cref{sec:nonnesting}.

\subsection{Generalizations to Finite Coxeter Groups}
A.~Bj\"orner and M.~Wachs' observation that the Tamari lattice arises as a sublattice of the weak order on $\mathfrak{S}_{n}$ was the precursor to N.~Reading's definition of the \defn{Cambrian lattices}.  Fixing a Coxeter element $c$, such a lattice may be described as the restriction of the weak order of a finite Coxeter group to certain \defn{$c$-aligned elements}, which---as with $231$-avoiding permutations---are characterized by their inversion sets~\cite{reading06cambrian}.  The Cambrian lattices thus naturally generalize the Tamari lattice to \textit{any} finite Coxeter group and \textit{any} Coxeter element.

N.~Reading's aligned elements of a Coxeter group $W$ have a surprisingly different characterization using reduced words, and in the guise of \defn{sortable elements} they provide a bridge between two famous families of objects attached to a finite Coxeter group: the \defn{$W$-noncrossing partitions} and the \defn{$W$-clusters}~\cite{reading07clusters}.  Remarkably, these objects are uniformly enumerated by a simple product formula depending on the degrees of $W$~\cite{reiner97non}*{Remark~2}.

In the second part of this article, we define analogues of the sets $\mathfrak{S}_{n}^{J}(231)$, $\NC_{n}^{J}$ and $\NN_{n}^{J}$ for all finite Coxeter groups, and we study in which cases we retain the property that these sets are equinumerous.  We present computational evidence that this is the case for groups of ``coincidental'' type $A_{n}$, $B_{n}$, $H_{3}$, and $I_{2}(m)$ (these are exactly those types for which every wall of the Coxeter complex is again a Coxeter complex).

\subsection{Further Generalizations}
The key idea in the definition of parabolic alignment is a certain forcing of inversions with respect to the root order of a particular reduced expression for the longest element in the parabolic quotient.  In the last part of this article we generalize this idea to any reduced expression of any element of any (not necessarily finite) Coxeter group.  This generalization comes at the price of losing the lattice property.

\subsection{Outline of the Paper}
This article is structured as follows.  We recall the basic notions for the symmetric group and its parabolic quotients in~\Cref{sec:symmetric_group}.  In~\Cref{sec:tamari_pq} we define $(J,231)$-avoiding permutations and characterize them in terms of their inversion sets.  We prove~\Cref{thm:parabolic_tamari_lattice} in ~\Cref{sec:parabolic_tamari_quotient}. ~\Cref{sec:noncrossing,sec:nonnesting} introduce noncrossing partitions and nonnesting partitions for parabolic quotients of the symmetric group, and culminate in the proof of~\Cref{thm:parabolic_bijections}.  \Cref{sec:generalizations} is concerned with the generalization of $231$-avoiding permutations, noncrossing and nonnesting partitions to parabolic quotients of arbitrary finite Coxeter groups.  We also propose a definition of a parabolic Coxeter-Catalan number for the coincidental types at the end of~\Cref{sec:numerology}.  In~\Cref{sec:aligned_expressions} we generalize the definition of alignment to any reduced expression of any element in any (not-necessarily finite) Coxeter group.

\section{The Symmetric Group}
  \label{sec:symmetric_group}

In this section, we recall the definitions of weak order, $231$-avoiding permutations, and parabolic quotients of the symmetric group.

\subsection{Weak Order}
  \label{sec:weak_order_symmetric_group}

The \defn{symmetric group} $\mathfrak{S}_{n}$ is the group of permutations of $[n]:=\{1,2,\ldots,n\}$. Let $S:=\{s_{1},s_{2},\ldots,s_{n-1}\}$ denote the set of \defn{adjacent transpositions} of $\mathfrak{S}_n$, \ie $s_{i}:=(i,i\!+\!1)$ for $i\in[n-1]$. It is well known that $\mathfrak{S}_{n}$ is isomorphic to the Coxeter group $A_{n-1}$, and so admits a presentation of the form
\begin{equation}\label{eq:symmetric_presentation}
  \mathfrak{S}_{n}=\left\langle S \mid s_{i}^{2}=(s_{i}s_{j})^2=(s_{i}s_{i+1})^{3}=e,\;\text{for}\;\lvert i-j\rvert>1 \right \rangle
\end{equation}
where $\id$ denotes the identity permutation.  We may specify a permutation $w \in \mathfrak{S}_n$ using one-line notation: $w=w_{1}w_{2}\ldots w_{n}$, where $w_{i}=w(i)$ for $i\in [n]$.  Its \defn{inversion set} is defined by
\begin{displaymath}
  \inv(w) := \bigl\{(i,j)\mid 1\leq i<j\leq n\;\text{and}\;w_{i}>w_{j}\bigr\}.
\end{displaymath}
The \defn{(left) weak order} is the partial order on $\mathfrak{S}_{n}$ defined by $u\leq_{S}v$ if and only if $\inv(u)\subseteq\inv(v)$; and we denote by $\Weak(\mathfrak{S}_{n})$ the partially ordered set $(\mathfrak{S}_{n},\leq_{S})$.  The \defn{cover relations} of $\Weak(\mathfrak{S}_{n})$ are relations $u\leq_{S}v$ such that $\inv(v)\setminus\inv(u)=\bigl\{(i,j)\bigr\}$ with $v_{i}=v_{j}+1$.  We usually write $u\lessdot_{S}v$ in such a case.  The poset $\Weak(\mathfrak{S}_n)$ is a lattice~\cite{bjorner05combinatorics}*{Theorem~3.2.1}, so that any two elements have a greatest lower bound and a least upper bound.  In particular, there is a unique maximal element $\wo$ in $\Weak(\mathfrak{S}_{n})$ whose one-line notation is $\wo=n(n-1)\ldots 1$.  We refer the reader to~\cite{bjorner05combinatorics}*{Section~3} for more background on the weak order, in the broader context of Coxeter groups.

A permutation $w\in\mathfrak{S}_{n}$ is \defn{$231$-avoiding} if there exists no triple $i<j<k$ such that $w_{k}<w_{i}<w_{j}$.  Let $\mathfrak{S}_{n}(231)$ denote the set of $231$-avoiding permutations of $\mathfrak{S}_{n}$.  Lemma~9.8 in \cite{bjorner97shellable} implies that the $231$-avoiding permutations can be characterized by their inversion sets.  More precisely, $w\in\mathfrak{S}_{n}$ is $231$-avoiding if and only if its inversion set is \defn{compressed}, \ie if $i<j<k$ and $(i,k)\in\inv(w)$, then $(i,j)\in\inv(w)$.

\begin{remark}\label{rem:foreshadowing}
  Let $w\in\mathfrak{S}_{n}$ and choose $i<j<k$.  It is immediate to verify that whenever $(i,k)\in\inv(w)$, then we also have $(i,j)\in\inv(w)$ or $(j,k)\in\inv(w)$ (or both).  We may interpret the property that $\inv(w)$ is compressed as stating that $\inv(w)$ is aligned with respect to the lexicographic order on all transpositions.  This perspective foreshadows N.~Reading's definition of aligned elements in a Coxeter group~\cite{reading07clusters}*{Section~4}.  We generalize this notion in ~\Cref{def:parabolic_alignment,def:general_alignment}.
\end{remark}

The next result identifies the Tamari lattice $\Tamari_n$ as the subposet of the weak order on $\mathfrak{S}_{n}$ induced by the $231$-avoiding permutations.  The reader may take this as the definition of $\Tamari_n$.

\begin{theorem}[\cite{bjorner97shellable}*{Theorem~9.6(ii)}]\label{thm:231_avoiding_tamari}
  For $n>0$ the poset $\Weak\bigl(\mathfrak{S}_{n}(231)\bigr)$ is isomorphic to the Tamari lattice $\Tamari_{n}$.
\end{theorem}

\subsection{Parabolic Quotients}
  \label{sec:parabolic_quotients}

Any subset $J\subseteq S$ naturally generates a subgroup of $\mathfrak{S}_{n}$ isomorphic to a direct product of symmetric groups of smaller rank.  We call such a subgroup \defn{parabolic}, and we denote it by $(\mathfrak{S}_{n})_{J}$.  We define the \defn{parabolic quotient} of $\mathfrak{S}_{n}$ with respect to $J$ by
\begin{displaymath}
  \mathfrak{S}_{n}^{J} := \{w\in\mathfrak{S}_{n}\mid w<_{S}ws\;\text{for all}\;s\in J\}.
\end{displaymath}
The set $\mathfrak{S}_{n}^{J}$ thus consists of the minimal length representatives of the right cosets of the corresponding parabolic subgroup.  By~\cite{bjorner05combinatorics}*{Proposition~2.4.4}, any permutation $w\in\mathfrak{S}_{n}$ can be uniquely written as $w=w^{J}\cdot w_{J}$, for some $w^{J}\in\mathfrak{S}_{n}^{J}$ and $w_{J}\in(\mathfrak{S}_{n})_{J}$.  In particular, the maximal element $\wo\in \mathfrak{S}_n$ is written as $\wo=\wo^J \cdot (\wo)_J,$ where $\wo^J$ is the longest element of $\mathfrak{S}_{n}^{J}$ and $(\wo)_J$ is the longest element of $(\mathfrak{S}_n)_J$.

Since any element of $\mathfrak{S}_{n}^{J}$ is itself a permutation, we may consider  $\Weak(\mathfrak{S}_{n}^{J})$---the restriction of the weak order on $\mathfrak{S}_n$ to the parabolic quotient.  It follows from \cite{bjorner88generalized}*{Theorem~4.1} that this poset is isomorphic to the weak order interval $[\id,\wo^{J}]$, and is therefore a lattice.

\section{Tamari Lattices for Parabolic Quotients of $\mathfrak{S}_n$}
  \label{sec:tamari_pq}

In this section, we define the set $\mathfrak{S}_{n}^{J}(231)$ of $231$-avoiding permutations of the parabolic quotient $\mathfrak{S}_{n}^{J}$.  We characterize the inversion sets of the permutations in $\mathfrak{S}_n^J(231)$ and prove that $\Weak\bigl(\mathfrak{S}_n^J(231)\bigr)$ is a lattice.

\subsection{$231$-Avoidance}
  \label{sec:j_aligned}

Let $J:=S\setminus\{s_{j_{1}},s_{j_{2}},\ldots,s_{j_{r}}\}$, and let $\B(J)$ be the set partition of $[n]$ given by the parts
\begin{displaymath}
  \bigl\{\{1,\ldots,j_{1}\},\{j_{1}+1,\ldots,j_{2}\},\ldots,\{j_{r-1}+1,\ldots,j_{r}\},\{j_{r}+1,\ldots,n\}\bigr\}.
\end{displaymath}
We call the parts of $\B(J)$ the \defn{$J$-regions}. We indicate the parts of $\B(J)$ occuring in the one-line notation of a permutation $w\in\mathfrak{S}_{n}^{J}$ by vertical bars.

\begin{lemma}\label{lem:increasing_regions}
  If $w\in\mathfrak{S}_{n}^{J}$, then the one-line notation of $w$ has the form
  \begin{displaymath}
    w = w_{1}<\cdots<w_{j_{1}}\mid w_{j_{1}+1}<\cdots<w_{j_{2}}\mid\cdots\mid w_{j_{r}+1}<\cdots<w_{n}.
  \end{displaymath}
\end{lemma}
\begin{proof}
  If this were not the case, then there would be some index $i\in[n]$ such that $w_{i}>w_{i+1}$, and $j_{l}+1<i<j_{l+1}-1$ for some $l\in\{0,1,\ldots,r\}$, where $j_{0}=1$ and $j_{r+1}=n$. It follows that $(i,i+1)\in\inv(w)$ and $(i,i+1)\notin\inv(ws_{i})$. By definition, it follows that $ws_{i}<_{S}w$, which contradicts the assumption that $w\in\mathfrak{S}_{n}^{J}$, since $s_{i}\in J$.
\end{proof}

\begin{definition}\label{def:parabolic_231}
    A permutation $w\in\mathfrak{S}_{n}^{J}$ contains a \defn{($J$,$231$)-pattern} if there exist three indices $i<j<k$, all of which lie in different $J$-regions, such that $w_{k}<w_{i}<w_{j}$ and $w_{i}=w_{k}+1$.  We say that $w$ is \defn{($J$,$231$)-avoiding} if it does not contain a ($J$,$231$)-pattern; and we denote the set of all ($J$,$231$)-avoiding permutations of $\mathfrak{S}_{n}^{J}$ by $\mathfrak{S}_{n}^{J}(231)$.
\end{definition}

\begin{figure}
  \centering
  \begin{tikzpicture}
    \draw(-1.15,-2) -- (11.6,-2) -- (11.6,4) -- (-1.15,4) -- (-1.15,-2);
    \draw(3.15,-2) -- (3.15,4);
    \draw(7.35,-2) -- (7.35,4);
    \draw(1,1) node{\begin{tikzpicture}\small
        \def\x{1.5};
        \def\y{1};
        \draw(2*\x,1*\y) node[fill=gray!50!white](n1){$1|23|4$};
        \draw(1.5*\x,2*\y) node[fill=gray!50!white](n2){$2|13|4$};
        \draw(2.5*\x,2*\y) node[fill=gray!50!white](n3){$1|24|3$};
        \draw(1*\x,3*\y) node[fill=gray!50!white](n4){$3|12|4$};
        \draw(2*\x,3*\y) node[fill=gray!50!white](n5){$2|14|3$};
        \draw(3*\x,3*\y) node(n6)[fill=gray!50!white]{$1|34|2$};
        \draw(1*\x,4*\y) node(n7)[fill=gray!50!white]{$4|12|3$};
        \draw(2*\x,4*\y) node(n8){$3|14|2$};
        \draw(3*\x,4*\y) node(n9){$2|34|1$};
        \draw(1.5*\x,5*\y) node(n10)[fill=gray!50!white]{$4|13|2$};
        \draw(2.5*\x,5*\y) node(n11)[fill=gray!50!white]{$3|24|1$};
        \draw(2*\x,6*\y) node(n12)[fill=gray!50!white]{$4|23|1$};
        \draw(n1) -- (n2);
        \draw(n1) -- (n3);
        \draw(n2) -- (n4);
        \draw(n2) -- (n5);
        \draw(n3) -- (n5);
        \draw(n3) -- (n6);
        \draw(n4) -- (n7);
        \draw(n5) -- (n8);
        \draw(n6) -- (n9);
        \draw(n7) -- (n10);
        \draw(n8) -- (n10);
        \draw(n8) -- (n11);
        \draw(n9) -- (n11);
        \draw(n10) -- (n12);
        \draw(n11) -- (n12);
      \end{tikzpicture}};
    \draw(5.25,1) node{\begin{tikzpicture}\small
        \def\x{1.5};
        \def\y{1};
        \draw(2*\x,1*\y) node(n1){\begin{tikzpicture}\tiny
            \ncbody
          \end{tikzpicture}};
        \draw(1.5*\x,2*\y) node(n2){\begin{tikzpicture}\tiny
            \ncbody
            \draw(p1) .. controls (1.1*\xx,.9) and (1.4*\xx,.9) .. (1.5*\xx,1) .. controls (1.6*\xx,1.1) and (1.9*\xx,1.1) .. (p2);
          \end{tikzpicture}};
        \draw(2.5*\x,2*\y) node(n3){\begin{tikzpicture}\tiny
            \ncbody
            \draw(p3) .. controls (3.1*\xx,.9) and (3.4*\xx,.9) .. (3.5*\xx,1) .. controls (3.6*\xx,1.1) and (3.9*\xx,1.1) .. (p4);
          \end{tikzpicture}};
        \draw(1*\x,3*\y) node(n4){\begin{tikzpicture}\tiny
            \ncbody
            \draw(p1) .. controls (1.1*\xx,.9) and (1.4*\xx,.9) .. (1.5*\xx,1) .. controls (1.9*\xx,1.15) and (2.6*\xx,1.15) .. (p3);
          \end{tikzpicture}};
        \draw(2*\x,3*\y) node(n5){\begin{tikzpicture}\tiny
            \ncbody
            \draw(p1) .. controls (1.1*\xx,.9) and (1.4*\xx,.9) .. (1.5*\xx,1) .. controls (1.6*\xx,1.1) and (1.9*\xx,1.1) .. (p2);
            \draw(p3) .. controls (3.1*\xx,.9) and (3.4*\xx,.9) .. (3.5*\xx,1) .. controls (3.6*\xx,1.1) and (3.9*\xx,1.1) .. (p4);
          \end{tikzpicture}};
        \draw(3*\x,3*\y) node(n6){\begin{tikzpicture}\tiny
            \ncbody
            \draw(p2) .. controls (2.4*\xx,.85) and (3.1*\xx,.85) .. (3.5*\xx,1) .. controls (3.6*\xx,1.1) and (3.9*\xx,1.1) .. (p4);
          \end{tikzpicture}};
        \draw(1*\x,4*\y) node(n7){\begin{tikzpicture}\tiny
            \ncbody
            \draw(p1) .. controls (1.1*\xx,.9) and (1.4*\xx,.9) .. (1.5*\xx,1) .. controls (2.2*\xx,1.2) and (3.3*\xx,1.2) .. (p4);
          \end{tikzpicture}};
        \draw(1.5*\x,5*\y) node(n10){\begin{tikzpicture}\tiny
            \ncbody
            \draw(p1) .. controls (1.1*\xx,.9) and (1.4*\xx,.9) .. (1.5*\xx,1) .. controls (1.9*\xx,1.15) and (2.6*\xx,1.15) .. (p3);
            \draw(p3) .. controls (3.1*\xx,.9) and (3.4*\xx,.9) .. (3.5*\xx,1) .. controls (3.6*\xx,1.1) and (3.9*\xx,1.1) .. (p4);
          \end{tikzpicture}};
        \draw(2.5*\x,5*\y) node(n11){\begin{tikzpicture}\tiny
            \ncbody
            \draw(p1) .. controls (1.1*\xx,.9) and (1.4*\xx,.9) .. (1.5*\xx,1) .. controls (1.6*\xx,1.15) and (1.9*\xx,1.15) .. (p2);
            \draw(p2) .. controls (2.4*\xx,.85) and (3.1*\xx,.85) .. (3.5*\xx,1) .. controls (3.6*\xx,1.1) and (3.9*\xx,1.1) .. (p4);
          \end{tikzpicture}};
        \draw(2*\x,6*\y) node(n12){\begin{tikzpicture}\tiny
            \ncbody
            \draw(p1) .. controls (1.1*\xx,.9) and (1.4*\xx,.9) .. (1.5*\xx,1) .. controls (1.9*\xx,1.15) and (2.6*\xx,1.15) .. (p3);
            \draw(p2) .. controls (2.4*\xx,.85) and (3.1*\xx,.85) .. (3.5*\xx,1) .. controls (3.6*\xx,1.1) and (3.9*\xx,1.1) .. (p4);
          \end{tikzpicture}};
        \draw(n1) -- (n2);
        \draw(n1) -- (n3);
        \draw(n2) -- (n4);
        \draw(n2) -- (n5);
        \draw(n3) -- (n5);
        \draw(n3) -- (n6);
        \draw(n4) -- (n7);
        \draw(n5) -- (2*\x,4*\y);
        \draw(n6) -- (3*\x,4*\y);
        \draw(n7) -- (n10);
        \draw(2*\x,4*\y) -- (n10);
        \draw(2*\x,4*\y) -- (n11);
        \draw(3*\x,4*\y) -- (n11);
        \draw(n10) -- (n12);
        \draw(n11) -- (n12);
      \end{tikzpicture}};
    \draw(9.5,1) node{\begin{tikzpicture}\small
        \def\x{1.5};
        \def\y{1};
        \draw(2*\x,1*\y) node(n1){\begin{tikzpicture}\tiny
            \nnbody
          \end{tikzpicture}};
        \draw(1.5*\x,2*\y) node(n2){\begin{tikzpicture}\tiny
            \nnbody
            \draw(.4,1.4) node[circle,draw,fill=white,scale=.5]{};
          \end{tikzpicture}};
        \draw(2.5*\x,2*\y) node(n3){\begin{tikzpicture}\tiny
            \nnbody
            \draw(.5,1.2) node[circle,draw,fill=white,scale=.5]{};
            \draw(.4,1.4) node[circle,draw,fill=white,scale=.5]{};
          \end{tikzpicture}};
        \draw(1*\x,3*\y) node(n4){\begin{tikzpicture}\tiny
            \nnbody
            \draw(.3,1.2) node[circle,draw,fill=white,scale=.5]{};
            \draw(.4,1.4) node[circle,draw,fill=white,scale=.5]{};
          \end{tikzpicture}};
        \draw(2*\x,3*\y) node(n5){\begin{tikzpicture}\tiny
            \nnbody
            \draw(.3,1.2) node[circle,draw,fill=white,scale=.5]{};
            \draw(.5,1.2) node[circle,draw,fill=white,scale=.5]{};
            \draw(.4,1.4) node[circle,draw,fill=white,scale=.5]{};
          \end{tikzpicture}};
        \draw(3*\x,3*\y) node(n6){\begin{tikzpicture}\tiny
            \nnbody
            \draw(.6,1) node[circle,draw,fill=white,scale=.5]{};
            \draw(.5,1.2) node[circle,draw,fill=white,scale=.5]{};
            \draw(.4,1.4) node[circle,draw,fill=white,scale=.5]{};
          \end{tikzpicture}};
        \draw(1*\x,4*\y) node(n7){\begin{tikzpicture}\tiny
            \nnbody
            \draw(.2,1) node[circle,draw,fill=white,scale=.5]{};
            \draw(.3,1.2) node[circle,draw,fill=white,scale=.5]{};
            \draw(.4,1.4) node[circle,draw,fill=white,scale=.5]{};
          \end{tikzpicture}};
        \draw(1.5*\x,5*\y) node(n10){\begin{tikzpicture}\tiny
            \nnbody
            \draw(.2,1) node[circle,draw,fill=white,scale=.5]{};
            \draw(.3,1.2) node[circle,draw,fill=white,scale=.5]{};
            \draw(.5,1.2) node[circle,draw,fill=white,scale=.5]{};
            \draw(.4,1.4) node[circle,draw,fill=white,scale=.5]{};
          \end{tikzpicture}};
        \draw(2.5*\x,5*\y) node(n11){\begin{tikzpicture}\tiny
            \nnbody
            \draw(.6,1) node[circle,draw,fill=white,scale=.5]{};
            \draw(.3,1.2) node[circle,draw,fill=white,scale=.5]{};
            \draw(.5,1.2) node[circle,draw,fill=white,scale=.5]{};
            \draw(.4,1.4) node[circle,draw,fill=white,scale=.5]{};
          \end{tikzpicture}};
        \draw(2*\x,6*\y) node(n12){\begin{tikzpicture}\tiny
            \nnbody
            \draw(.2,1) node[circle,draw,fill=white,scale=.5]{};
            \draw(.6,1) node[circle,draw,fill=white,scale=.5]{};
            \draw(.3,1.2) node[circle,draw,fill=white,scale=.5]{};
            \draw(.5,1.2) node[circle,draw,fill=white,scale=.5]{};
            \draw(.4,1.4) node[circle,draw,fill=white,scale=.5]{};
          \end{tikzpicture}};
        \draw(n1) -- (n2);
        \draw(n1) -- (n3);
        \draw(n2) -- (n4);
        \draw(n2) -- (n5);
        \draw(n3) -- (n5);
        \draw(n3) -- (n6);
        \draw(n4) -- (n7);
        \draw(n5) -- (2*\x,4*\y);
        \draw(n6) -- (3*\x,4*\y);
        \draw(n7) -- (n10);
        \draw(2*\x,4*\y) -- (n10);
        \draw(2*\x,4*\y) -- (n11);
        \draw(3*\x,4*\y) -- (n11);
        \draw(n10) -- (n12);
        \draw(n11) -- (n12);
      \end{tikzpicture}};
  \end{tikzpicture}
  \caption{The parabolic Tamari lattice $\Tamari_4^{\{s_2\}}$.  The poset on the left has every permutation of the parabolic quotient $\mathfrak{S}_4^{\{s_2\}}$ with the $\bigl(\{s_{2}\}$,$231\bigr)$-avoiding elements marked in gray; the posets in the middle and on the right are labeled by the $\{s_{2}\}$-noncrossing partitions and the $\{s_{2}\}$-nonnesting partitions of a four-element set, respectively.}
  \label{fig:j_2_tamari}
\end{figure}
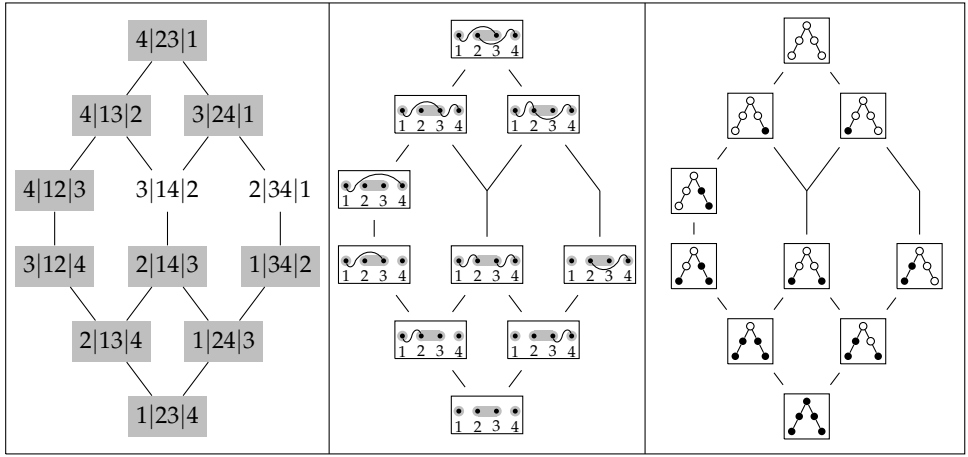

\begin{example} \label{eq:j_2}
    The left image in~\Cref{fig:j_2_tamari} shows $\Weak\bigl(\mathfrak{S}_{4}^{\{s_{2}\}}\bigr)$, where the $\bigl(\{s_{2}\}$,$231\bigr)$-avoiding permutations have been shaded in gray.  Notice that the longest permutation $4|23|1$ is not $231$-avoiding, since it contains the subsequence $231$.  However, since the $2$ and the $3$ lie in the same $\{s_{2}\}$-region, this sequence does not form an $\bigl(\{s_{2}\}$,$231\bigr)$-pattern.
\end{example}

The following \Cref{lem:j_231_well_defined} shows that \Cref{def:parabolic_231} is a generalization of $231$-avoiding permutations by showing that $\mathfrak{S}_{n}^{\emptyset}(231)=\mathfrak{S}_{n}(231)$.

\begin{proposition}\label{lem:j_231_well_defined}
  If $w\in\mathfrak{S}_{n}$ has a $231$-pattern, then there exist indices $i<j<k$ such that $w_{k}<w_{i}<w_{j}$ and $w_{i}=w_{k}+1$. Consequently, $\mathfrak{S}_n^\emptyset(231) = \mathfrak{S}_n(231)$.
\end{proposition}
\begin{proof}
  Let $i<j<k$ be indices such that $w_{k}<w_{i}<w_{j}$, and choose them in such a way that $w_{i}-w_{k}$ is minimal. We claim that in this case $w_{i}=w_{k}+1$.  Assume the opposite. Then there is some $d\in [n]$ with $w_{k}<w_{d}<w_{i}$. If $d<j$, then $(d,j,k)$ forms a $231$-pattern in $w$.  But $w_{d}-w_{k}<w_{i}-w_{k}$, which contradicts the choice of $(i,j,k)$. If $d>j$, then $(i,j,d)$ forms a $231$-pattern in $w$.  But $w_{i}-w_{d}<w_{i}-w_{k}$, which again contradicts the choice of $(i,j,k)$.
\end{proof}

\begin{remark}
  After an extended abstract of this paper appeared in~\cite{muehle15tamari}, R.~Proctor and M.~Willis gave a different definition of parabolic pattern avoidance~\cite{proctor18convexity}.  More precisely, if $R=\{j_{1},j_{2},\ldots,j_{r}\}$, then they say that $w\in\mathfrak{S}_{n}^{J}$ is \defn{$R$-$312$-containing} if there exists $h\in[r-1]$ and indices $a,b,c\in[n]$ with $a\leq j_{h}<b\leq j_{h+1}<c$ such that $w_{b}<w_{c}<w_{a}$.  Any element of $\mathfrak{S}_{n}^{J}$ that is not $R$-$312$-containing is \defn{$R$-$312$-avoiding}.  They suggested the term ``parabolic Catalan number'' for the cardinality of the set of all $R$-$312$-avoiding permutations, and exhibit several combinatorial objects associated with $\mathfrak{S}_{n}^{J}$ that are enumerated by these numbers~\cites{proctor17parabolic1,proctor17parabolic2,proctor18convexity}.

  It is straightforward to define $R$-$231$-avoiding permutations in the sense of R.~Proctor and M.~Willis, but this definition is more restrictive than our~\Cref{def:parabolic_231}.  For example, the permutation $3|24|1\in\mathfrak{S}_{4}^{\{s_{2}\}}$ is $\{1,3\}$-$231$-containing, since the '3', the '4', and the '1' form a parabolic $231$-pattern.  However, the '3' and the '1' do not form a descent, and this permutation turns out to be $(\{s_{2}\},231)$-avoiding.
\end{remark}

\subsection{Compressed Inversion Sets}
We now generalize A.~Bj{\"o}rner and M.~Wachs' definition of compressed inversion sets to parabolic quotients.  Define the \defn{descent set} of $w$ by
\begin{displaymath}
  \des(w) := \bigl\{(i,j)\in\inv(w)\mid w_{i}=w_{j}+1\bigr\}.
\end{displaymath}

\begin{definition}\label{def:parabolic_compressed}
    An inversion set $\inv(w)$ for a permutation $w \in \mathfrak{S}_n^J$ is \defn{$J$-compressed} if whenever there are three indices $i<j<k$, each in different $J$-regions and such that $(i,k)\in\des(w)$, it follows that $(i,j)\in\inv(w)$.
\end{definition}

\begin{lemma}\label{lem:j_compressed_avoiding}
  A permutation $w\in\mathfrak{S}_{n}^J$ is ($J$,$231$)-avoiding if and only if $\inv(w)$ is $J$-compressed.
\end{lemma}
\begin{proof}
  Suppose first that $w\in\mathfrak{S}_{n}^{J}$ is not ($J$,$231$)-avoiding.  By definition there exist indices $i<j<k$ each in different $J$-regions such that $w_{k}<w_{i}<w_{j}$ as well as $w_{i}=w_{k}+1$.  This means that $(i,k)\in\des(w)$ but $(i,j)\notin\inv(w)$, which implies that $\inv(w)$ is not $J$-compressed.

  On the other hand, suppose that $w$ is ($J$,$231$)-avoiding, and choose three indices $i<j<k$, each in different $J$-regions and such that $(i,k)\in\des(w)$.  Since $w$ does not contain a ($J$,$231$)-pattern we must have $w_{j}<w_{i}$, which implies $(i,j)\in\inv(w)$.  Hence $\inv(w)$ is $J$-compressed.
\end{proof}

\subsection{Tamari Lattices for Parabolic Quotients}
  \label{sec:tamari}

Write $\Weak\bigl(\mathfrak{S}_n^J(231)\bigr)$ for the restriction of the weak order on the parabolic quotient $\mathfrak{S}_n^J$ to the $(J,231)$-avoiding permutations.  We now prove the first part of~\Cref{thm:parabolic_tamari_lattice}---that $\Weak\bigl(\mathfrak{S}_n^J(231)\bigr)$ is a lattice.  The proof follows from the next lemma, which is modeled after \cite{reading06cambrian}*{Lemma~5.6} for the case $J=\emptyset$.

\begin{lemma}\label{lem:down_projection}
  For every $w\in\mathfrak{S}_{n}^{J}$, there is a unique $w'\in\mathfrak{S}_{n}^{J}(231)$ such that $\inv(w')$ is the maximal set under containment among all $J$-compressed inversion sets $\inv(u)\subseteq\inv(w)$.
\end{lemma}

\begin{proof}
  We proceed by induction on the cardinality of $\inv(w)$. If $\inv(w)=\emptyset$, then $\inv(w)$ is $J$-compressed and the claim holds trivially. Suppose that $\bigl\lvert\inv(w)\bigr\rvert=r$, and that the claim is true for all $x\in\mathfrak{S}_{n}^{J}$ with $\bigl\lvert\inv(x)\bigr\rvert<r$.

  If $\inv(w)$ is already $J$-compressed, then we set $w'=w$ and we are done. Otherwise,~\Cref{lem:j_compressed_avoiding} implies that $w$ contains an instance of a ($J$,$231$)-pattern, which means that there are indices $i<j<k$ that all lie in different $J$-regions such that $w_{k}<w_{i}<w_{j}$ and $w_{i}=w_{k}+1$. Consider the lower cover $u$ of $w$ in which $w_{i}$ and $w_{k}$ are exchanged.  In particular, we have $\inv(w)=\inv(u)\cup\bigl\{(i,k)\bigr\}$.  By the induction hypothesis, there exists some $u'\in\mathfrak{S}_{n}^{J}$ such that $\inv(u')$ is the unique maximal $J$-compressed inversion set that is contained in $\inv(u)$. We claim that $w'=u'$. 

  In order to prove this claim, we choose some element $v\in\mathfrak{S}_{n}^{J}$ such that $\inv(v)$ is $J$-compressed and $\inv(v)\subseteq\inv(w)$. By construction, we have $(i,j)\notin\inv(w)$, and hence $(i,j)\notin\inv(v)$. Since $\inv(v)$ is $J$-compressed it follows by definition that $v_{i}\neq v_{k}+1$. We want to show that $\inv(v)\subseteq\inv(u)$, which amounts to showing that $(i,k)\notin\inv(v)$ because $\inv(w)\setminus\inv(u)=\bigl\{(i,k)\bigr\}$.

  We assume the opposite, and in view of the argument above it follows that $v_{i}>v_{k}+1$. Let $d$ be the index such that $v_{d}=v_{k}+1$, and let $e$ be the index such that $v_{i}=v_{e}+1$. Since $w_{i}=w_{k}+1$, we observe the following:
  \begin{align}
    \tag{D}\text{either} & \quad w_{d}<w_{k}\quad\text{or}\quad w_{i}<w_{d},\label{eq:prop_d}\quad\text{and}\\
    \tag{E}\text{either} & \quad w_{e}<w_{k}\quad\text{or}\quad w_{i}<w_{e}.\label{eq:prop_e}
  \end{align}
  We have the following relations:
  \begin{displaymath}
    v_{j}>v_{i}>v_{e}\geq v_{d}>v_{k}.
  \end{displaymath}
  (If $v_{j}<v_{i}$, then $(i,j)\in\inv(v)\subseteq\inv(w)$, which is a contradiction.) We now distinguish five cases.

  (i) Let $d<i<k$.  Then $(d,k)\in\des(v)\subseteq\inv(w)$.  It follows that $w_{d}>w_{k}$, and \eqref{eq:prop_d} implies $w_{d}>w_{i}$.  \Cref{lem:increasing_regions} implies that $d$ and $i$ lie in different $J$-regions.  Since $\inv(v)$ is $J$-compressed, we conclude $(d,i)\in\inv(v)$.  Hence $v_{i}<v_{d}=v_{k}+1<v_{i}$, which is a contradiction.

  (ii) Let $i<d<k$.  Then $(i,d),(d,k)\in\inv(v)\subseteq\inv(w)$.  It follows that $w_{i}>w_{d}>w_{k}$, which contradicts $(i,k)\in\des(w)$.

  (iii) Let $i<e<k$.  Then $(i,e),(e,k)\in\inv(v)\subseteq\inv(w)$.  It follows that $w_{i}>w_{e}>w_{k}$, which contradicts $(i,k)\in\des(w)$.

  (iv) Let $i<k<e$.  Then $(i,e)\in\des(v)\subseteq\inv(w)$.  It follows that $w_{i}>w_{e}$, and \eqref{eq:prop_e} implies $w_{e}<w_{k}$.  \Cref{lem:increasing_regions} implies that $k$ and $e$ lie in different $J$-regions.  Since $\inv(v)$ is $J$-compressed, we conclude $(k,e)\in\inv(v)$.  Hence $v_{i}=v_{e}+1<v_{k}+1<v_{i}$, which is a contradiction.

  (v) Let $e<i<k<d$, which in particular implies that $(e,d)\in\inv(v)\subseteq\inv(w)$.  Moreover, $(e,k),(i,d)\in\des(v)\subseteq\inv(w)$.  It follows that $w_{i}>w_{d}$ as well as $w_{e}>w_{k}$. Now \eqref{eq:prop_d} and \eqref{eq:prop_e} imply $w_{d}<w_{k}$ and $w_{e}>w_{i}$, respectively.  \Cref{lem:increasing_regions} implies that $e,i,k$ and $d$ all lie in different $J$-regions.

  Let $e'$ be the smallest element in the $J$-region of $e$ such that $v_{e'}>v_{d}$, and let $d'$ be the largest element in the $J$-region of $d$ such that $v_{d'}<v_{e'}$. We record that $e'\leq e<i<j<k<d\leq d'$, and we proceed by induction on $v_{e'}-v_{d'}$. If $v_{e'}=v_{d'}+1$, then $(e',i),(e',j)\in\inv(v)$, since $\inv(v)$ is $J$-compressed. \Cref{lem:increasing_regions} implies that $v_{e}\geq v_{e'}>v_{j}>v_{i}=v_{e}+1$, which is a contradiction. If $v_{e'}>v_{d'}+1$, then there must be some index $f$ with $v_{e'}>v_{f}>v_{d'}+1$.  By construction we have $v_{i}=v_{e}+1>v_{e}\geq v_{e'}>v_{f}$.

  If $f<i$ and they do not lie in the same $J$-region, then we can consider the triple $(f,i,d')$, and obtain a contradiction by induction, since $v_{f}-v_{d'}<v_{e'}-v_{d'}$.  If $f>i$ and they do not lie in the same $J$-region, then we can consider the triple $(e',i,f)$, and obtain a contradiction by induction, since $v_{e'}-v_{f}<v_{e'}-v_{d'}$.  If $f$ and $i$ lie in the same $J$-region, then we have $f<i$.  We can consider the triple $(f,j,d')$, and obtain a contradiction by induction, since $v_{f}-v_{d'}<v_{e'}-v_{d'}$.

  We have thus shown that $(i,k)\notin\inv(v)$, which implies $\inv(v)\subseteq\inv(u)$. By the induction assumption it follows that $\inv(v)\subseteq\inv(u')$, which proves $w'=u'$.
\end{proof}

Using~\Cref{lem:j_compressed_avoiding}, we may reformulate~\Cref{lem:down_projection}: for every $w\in\mathfrak{S}_n^J$, there exists a unique maximal ($J$,$231$)-avoiding permutation $w'$ with $w'\leq_{S}w$.  The following definition gives us notation to refer to $w'$.

\begin{definition}\label{def:down_projection}
  We define the projection 
  \begin{align*}
    \Pi_{\downarrow}^{J}\colon\mathfrak{S}_{n}^{J}\to\mathfrak{S}_{n}^{J}(231),\quad w\mapsto w',
  \end{align*}
  where $w'$ is the unique maximal ($J$,$231$)-avoiding permutation below $w$.
\end{definition}

\begin{proposition}\label{prop:j_231_avoiding_lattice}
    The poset $\Weak\bigl(\mathfrak{S}_{n}^{J}(231)\bigr)$ is a lattice.
\end{proposition}
\begin{proof}
  Let $w_{1},w_{2}\in\mathfrak{S}_{n}^{J}(231)$.
  \Cref{lem:down_projection} implies that there exists a unique maximal element $u'\in\mathfrak{S}_{n}^{J}(231)$ with $u'\leq_{S}w_{1},w_{2}$, which necessarily must be the meet of $w_{1}$ and $w_{2}$ in $\Weak\bigl(\mathfrak{S}_{n}^{J}(231)\bigr)$.  Since $\inv(\wo^{J})$ contains all possible inversions, we have $\Pi_{\downarrow}^{J}(\wo^J)=\wo^J$. We have thus established that $\Weak\bigl(\mathfrak{S}_{n}^{J}(231)\bigr)$ is a finite meet-semilattice with greatest element $\wo^{J}$.  It is a classical lattice-theoretic result (see for instance \cite{gratzer11lattice}*{Exercise~1.27}) that this suffices to show that $\Weak\bigl(\mathfrak{S}_{n}^{J}(231)\bigr)$ is a lattice.
\end{proof}

\Cref{lem:j_231_well_defined} implies that the set $\mathfrak{S}_{n}^{\emptyset}(231)$ coincides with the set of all classical $231$-avoiding permutations of $\mathfrak{S}_{n}$, and \Cref{thm:231_avoiding_tamari} states that $\Weak\bigl(\mathfrak{S}_{n}^{\emptyset}(231)\bigr)$ is isomorphic to the Tamari lattice $\Tamari_{n}$. In view of \Cref{prop:j_231_avoiding_lattice}, we denote the poset $\Weak\bigl(\mathfrak{S}_{n}^{J}(231)\bigr)$ by $\Tamari_{n}^{J}$, and call it the \defn{parabolic Tamari lattice}.

\begin{remark}\label{rem:different_from_interval}
  Consider the parabolic subgroup $(\mathfrak{S}_{n})_{J}$, and let $(\wo)_{J}$ denote the longest permutation in this subgroup. The poset of all $231$-avoiding permutations in the interval $[\id,(\wo)_{J}]$ is just an interval in the Tamari lattice $\Tamari_{n}$.

  If we consider instead parabolic \emph{quotients}, then even though the elements in $\mathfrak{S}_{n}^{J}$ form the interval $[\id,\wo^{J}]$, the lattice $\Tamari_{n}^{J}$ is not an interval in $\Tamari_{n}$.  For example, $\Tamari_4^{\{s_{2}\}}$ is depicted in \Cref{fig:j_2_tamari}. Observe that the maximal element $\wo^{\{s_{2}\}}=4|23|1$ is \emph{not} $231$-avoiding.
\end{remark}

\begin{remark}\label{rem:not_sublattice}
  The lattice $\Tamari_{n}^{J}$ is not generally a sublattice of $[\id,\wo^{J}]$. Consider again the case when $n=4$ and $J=\{s_{2}\}$.  Then the meet of $w_{1}=4|13|2$ and $w_{2}=3|24|1$ in weak order is $3|14|2$, while their meet in $\Tamari_4^{\{s_{2}\}}$ is $2|14|3$.

  In certain special cases---for example, when $J = \emptyset$ or for certain $J=S\setminus\{s\}$---we do obtain sublattices.
\end{remark}

\subsection{Parabolic Tamari Lattices are Lattice Quotients}
  \label{sec:parabolic_tamari_quotient}
In this section, we prove that $\Tamari_n^J$ is a lattice quotient of $\Weak(\mathfrak{S}_n^J)$, completing the proof of \Cref{thm:parabolic_tamari_lattice}. Recall for instance from \cite{reading06cambrian}*{Section~3} that an equivalence relation $\Theta$ on a lattice is a \defn{lattice congruence} if and only if all equivalence classes are intervals, and the projections that map an element to the least or greatest element in its equivalence class, respectively, are both order-preserving.

Using the map $\Pi_{\downarrow}^{J}$ from~\Cref{def:down_projection}, we define a binary relation $\Theta$ on $\mathfrak{S}_{n}^{J}$ by
\begin{equation}\label{eq:tamari_congruence}
  (w,w')\in\Theta\quad\text{if and only if}\quad\Pi_{\downarrow}^{J}(w)=\Pi_{\downarrow}^{J}(w').
\end{equation}
It is immediate that $\Theta$ is an equivalence relation and $\Pi_{\downarrow}^{J}$ maps $w\in\mathfrak{S}_{n}^{J}$ to the least element in its equivalence class.

\begin{lemma}\label{lem:convex_fibers}
  The fibers of $\Pi_{\downarrow}^{J}$ are order-convex, \ie if $u\leq_{S}x\leq_{S}v$ and $\Pi_{\downarrow}^{J}(u)=\Pi_{\downarrow}^{J}(v)$, then $\Pi_{\downarrow}^{J}(u)=\Pi_{\downarrow}^{J}(x)$.
\end{lemma}
\begin{proof}
  Let $u'=\Pi_{\downarrow}^{J}(u)=\Pi_{\downarrow}^{J}(v)$ and $x'=\Pi_{\downarrow}^{J}(x)$.  Since $x'\leq_{S}x\leq_{S}v$, \Cref{lem:down_projection} implies $x'\leq_{S}u'$.  Moreover, since $u'\leq_{S}u\leq_{S}x$, \Cref{lem:down_projection} implies $u'\leq_{S}x'$.
\end{proof}

We claim that every equivalence class of $\Theta$ has a greatest element.  In view of \Cref{lem:convex_fibers} this would imply that the equivalence classes of $\Theta$ are intervals in $\Weak(\mathfrak{S}_n^J)$.  In order to describe these greatest elements, we say that a permutation $w\in\mathfrak{S}_{n}^{J}$ has a \defn{($J$,$132$)-pattern} if there are indices $i<j<k$ each in different $J$-regions such that $w_{i}<w_{k}<w_{j}$ and $w_{k}=w_{i}+1$.  We say $w\in\mathfrak{S}_{n}^{J}$ is \defn{($J$,$132$)-avoiding} if it does not have a ($J$,$132$)-pattern.  The proof of the following result is almost identical to the proof of \Cref{lem:down_projection}.

\begin{lemma}\label{lem:up_projection}
  For any $w\in\mathfrak{S}_{n}^{J}$, there is a unique minimal ($J$,$132$)-avoiding permutation $w'$ with $w\leq_{S}w'$.
\end{lemma}

We therefore obtain a map $\Pi_{\uparrow}^{J}\colon\mathfrak{S}_{n}\to\mathfrak{S}_{n}^{J}(132)$ that maps $w$ to the unique minimal ($J$,$132$)-avoiding permutation $w'$ above $w$.

\begin{lemma}\label{lem:projections_order_preserving}
  The maps $\Pi_{\downarrow}^{J}$ and $\Pi_{\uparrow}^{J}$ are order-preserving.
\end{lemma}
\begin{proof}
  We only prove this property for $\Pi_{\downarrow}^{J}$, since the result for $\Pi_{\uparrow}^{J}$ follows analogously.

  Let $u,v\in\mathfrak{S}_{n}^{J}$ with $u\leq_{S}v$. \Cref{lem:down_projection} states that $\Pi_{\downarrow}^{J}(u)\leq_{S}u\leq_{S}v$.  Since $\Pi_{\downarrow}^{J}(v)$ is maximal among all elements below $v$ with $J$-compressed inversion sets, it follows that $\Pi_{\downarrow}^{J}(u)\leq_S\Pi_{\downarrow}^{J}(v)$.
\end{proof}

\begin{lemma}\label{lem:fiber_moves}
  Let $u,v\in\mathfrak{S}_{n}^{J}$ with $u\lessdot_{S}v$. The following are equivalent.
  \begin{enumerate}[(i)]
    \item There are indices $i<j<k$, each in different $J$-regions, such that $v_{k}<v_{i}<v_{j}$, $v_{i}=v_{k}+1$, and $\inv(v)\setminus\inv(u)=\bigl\{(i,k)\bigr\}$.
    \item $\Pi_{\downarrow}^{J}(u)=\Pi_{\downarrow}^{J}(v)$.
    \item $\Pi_{\uparrow}^{J}(u)=\Pi_{\uparrow}^{J}(v)$.
  \end{enumerate}
\end{lemma}
\begin{proof}
  Let $u\lessdot_{S}v$.  By definition we have $\inv(v)\setminus\inv(u)=\bigl\{(i,k)\bigr\}$ with $v_{i}=v_{k}+1$.  Observe that $i<k$ implies that $i$ and $k$ belong to different $J$-regions.

  Suppose that (i) holds. By assumption there is a ($J$,$231$)-pattern in $v$, which is induced by the indices $i<j<k$.  In view of \Cref{lem:j_compressed_avoiding} we conclude that $\inv(v)$ is not $J$-compressed.  By construction $u$ is the lower cover of $v$ which has $v_{i}$ and $v_{k}$ exchanged.  The proof that $\Pi_{\downarrow}^{J}(u)=\Pi_{\downarrow}^{J}(v)$ now proceeds as in \Cref{lem:down_projection}.  This proves that (i) implies (ii).  An analogous argument using \Cref{lem:up_projection} proves that (i) implies (iii).

  Now suppose that (i) does not hold.  In other words, assume that for any $j\in[n]$ which satisfies $i<j<k$ and does not belong to the same $J$-region as $i$ or $k$ we have $v_{j}<v_{i}$.  In particular, $i$ and $k$ do not participate in any ($J$,$231$)-pattern of $v$, and the maximality of $\Pi_{\downarrow}^{J}(v)$ implies that $(i,k)\in\inv\bigl(\Pi_{\downarrow}^{J}(v)\bigr)$.  On the other hand, since $(i,k)\notin\inv(u)$ we conclude $(i,k)\notin\inv\bigl(\Pi_{\downarrow}^{J}(u)\bigr)$.  Hence $\Pi_{\downarrow}^{J}(u)<_{S}\Pi_{\downarrow}^{J}(v)$.  This proves that (ii) implies (i).  We also see that $i$ and $k$ do not participate in any ($J$,$132$)-pattern of $u$, and the minimality of $\Pi_{\uparrow}^{J}(u)$ implies $(i,k)\notin\inv\bigl(\Pi_{\uparrow}^{J}(u)\bigr)$.  On the other hand $(i,k)\in\inv(v)\subseteq\inv\bigl(\Pi_{\uparrow}^{J}(v)\bigr)$.  Hence $\Pi_{\uparrow}^{J}(u)<_{S}\Pi_{\uparrow}^{J}(v)$.  This proves that (iii) implies (i) and the proof is complete.
\end{proof}

\begin{lemma}\label{lem:up_down_moves}
    If $\Pi_{\downarrow}^{J}(u)=\Pi_{\downarrow}^{J}(v)$ for some $u,v\in\mathfrak{S}_{n}^{J}$, then $\Pi_{\uparrow}^{J}(u)=\Pi_{\uparrow}^{J}(v)$.
\end{lemma}
\begin{proof}
    Assume that $u\leq_{S}v$. If $u\lessdot_{S}v$, then the claim follows from \Cref{lem:fiber_moves}.  If $u<_{S}v$ do not form a cover relation, we find the desired equality by repeated application of \Cref{lem:fiber_moves} using \Cref{lem:convex_fibers}.

    Otherwise, suppose that $u$ and $v$ are incomparable.  Then, $\Pi_{\downarrow}^{J}(u\wedge v)=\Pi_{\downarrow}^{J}(u)$, since $\Pi_{\downarrow}^{J}(u)=\Pi_{\downarrow}^{J}(v)$ is the unique maximal ($J$,$231$)-avoiding permutation below both $u$ and $v$.  Since $u\wedge v\leq_{S}u$ and $u\wedge v\leq_{S}v$, we conclude $\Pi_{\uparrow}^{J}(u)=\Pi_{\uparrow}^{J}(u\wedge v)=\Pi_{\uparrow}^{J}(v)$ using the argument above.
\end{proof}

\begin{proposition}\label{prop:congruence}
    The equivalence relation $\Theta$ from \eqref{eq:tamari_congruence} is in fact a lattice congruence on $[\id,\wo^{J}]$, and the corresponding quotient lattice is $\Tamari_n^J$.
\end{proposition}
\begin{proof}
	\Cref{lem:down_projection} implies that the equivalence classes of $\Theta$ have a least element, and these minimal elements are precisely the elements of $\mathfrak{S}_{n}^{J}(231)$.  \Cref{lem:up_down_moves} implies together with \Cref{lem:up_projection} that equivalence classes have a greatest element, and \Cref{lem:convex_fibers} implies that the equivalence class $[w]_{\Theta}$ is in fact equal to the interval $\bigl[\Pi_{\downarrow}^{J}(w),\Pi_{\uparrow}^{J}(w)\bigr]$.  \Cref{lem:projections_order_preserving} now completes the proof.
\end{proof}

\begin{proof}[Proof of \Cref{thm:parabolic_tamari_lattice}]
    This follows from \Cref{prop:j_231_avoiding_lattice,prop:congruence}.
\end{proof}

\section{Parabolic Noncrossing Partitions}
  \label{sec:noncrossing}

In this section, we define the set $\NC_n^J$ of noncrossing partitions for parabolic quotients, and give an explicit bijection between $\NC_n^J$ and $\mathfrak{S}_n^J(231)$.

Recall that a set partition of $[n]$ is a collection $\Pf=\{P_{1},P_{2},\ldots,P_{s}\}$ of pairwise disjoint, nonempty subsets of $[n]$ with the property that their union is $[n]$. The elements $P_{i}$ of $\Pf$ are called the \defn{parts} of $\Pf$. A pair $(a,b)$ is a \defn{bump} of $\Pf$ if $a,b\in P_{i}$ for some $i\in[s]$ and there is no $c\in P_{i}$ with $a<c<b$.  Classically, a set partition is \defn{noncrossing} if it does not contain two bumps $(i_{1},j_{1})$ and $(i_{2},j_{2})$ such that $i_{1}<i_{2}<j_{1}<j_{2}$~\cite{kreweras72sur}.  We introduce the following generalization.

\begin{definition}\label{def:parabolic_noncrossing_set_partitions}
    A partition $\Pf$ of $[n]$ is \defn{$J$-noncrossing} if it satisfies the following three conditions.
    \begin{description}
      \item[NC1\label{it:nc1}] If $i$ and $j$ lie in the same $J$-region, then they are not contained in the same part of $\Pf$.
      \item[NC2\label{it:nc2}] If two distinct bumps $(i_{1},i_{2})$ and $(j_{1},j_{2})$ of $\Pf$ satisfy $i_{1}<j_{1}<i_{2}<j_{2}$, then either $i_{1}$ and $j_{1}$ lie in the same $J$-region or $i_{2}$ and $j_{1}$ lie in the same $J$-region.
      \item[NC3\label{it:nc3}] If two distinct bumps $(i_{1},i_{2})$ and $(j_{1},j_{2})$ of $\Pf$ satisfy $i_{1}<j_{1}<j_{2}<i_{2}$, then $i_{1}$ and $j_{1}$ lie in different $J$-regions.
    \end{description}
\end{definition}

We denote the set of all $J$-noncrossing set partitions of $[n]$ by $\NC^J_n$.  If $J=\emptyset$, then we recover the classical noncrossing set partitions.  We now introduce a combinatorial model for the $J$-noncrossing partitions. We draw $n$ dots, labeled by the numbers $1,2,\ldots,n$, on a straight line, and highlight the $J$-regions by grouping the corresponding dots together. For any bump $(i,j)$ in $\Pf\in\NC^J_n$, we draw an arc connecting the dots corresponding to $i$ and $j$, respectively, that passes below all dots corresponding to indices $k>i$ that lie in the same $J$-region as $i$, and above all other dots between $i$ and $j$. See the bottom left of \Cref{fig:j_nonnesting_example} for an illustration.

Let $\Pf\in\NC_n^J$ and define a binary relation $\vec{R}_{\Pf}$ on the parts of $\Pf$ by setting $(B,B')\in\vec{R}_{\Pf}$ if there exists a bump $(i_1,i_2)$ of $\Pf$ with $i_1,i_2\in B$ such that $i_1<\min B'<i_2$.  This relation is certainly acyclic, and can therefore be extended to an order relation by taking reflexive and transitive closures.  Let $\vec{O}_{\Pf}$ be the partially ordered set whose ground sets are the parts of $\Pf$, and whose order relation is the reflexive and transitive closure of $\vec{R}_{\Pf}$.

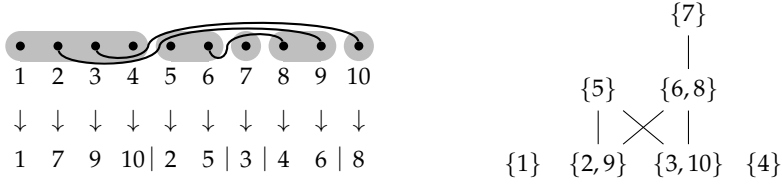
\begin{figure}
  \centering
  \begin{tikzpicture}
    \draw(1,1) node{
      \begin{tikzpicture}\small
        \def\x{.5};
        \def\y{1};
        \draw(2*\x,3*\y) node{};
        \draw(1*\x,2.5*\y) node[circle,fill,scale=.4](n1){};
        \draw(1*\x,2.1*\y) node{$1$};
        \draw(2*\x,2.5*\y) node[circle,fill,scale=.4](n2){};
        \draw(2*\x,2.1*\y) node{$2$};
        \draw(3*\x,2.5*\y) node[circle,fill,scale=.4](n3){};
        \draw(3*\x,2.1*\y) node{$3$};
        \draw(4*\x,2.5*\y) node[circle,fill,scale=.4](n4){};
        \draw(4*\x,2.1*\y) node{$4$};
        \draw(5*\x,2.5*\y) node[circle,fill,scale=.4](n5){};
        \draw(5*\x,2.1*\y) node{$5$};
        \draw(6*\x,2.5*\y) node[circle,fill,scale=.4](n6){};
        \draw(6*\x,2.1*\y) node{$6$};
        \draw(7*\x,2.5*\y) node[circle,fill,scale=.4](n7){};
        \draw(7*\x,2.1*\y) node{$7$};
        \draw(8*\x,2.5*\y) node[circle,fill,scale=.4](n8){};
        \draw(8*\x,2.1*\y) node{$8$};
        \draw(9*\x,2.5*\y) node[circle,fill,scale=.4](n9){};
        \draw(9*\x,2.1*\y) node{$9$};
        \draw(10*\x,2.5*\y) node[circle,fill,scale=.4](n10){};
        \draw(10*\x,2.1*\y) node{$10$};
        \draw[thick](n2) .. controls (2.1*\x,2.2*\y) and  (4.6*\x,2.2*\y) ..  (4.7*\x,2.5*\y) .. controls (4.8*\x,2.8*\y) and (8.9*\x,2.8*\y) .. (n9);
        \draw[thick](n3) .. controls (3.1*\x,2.3*\y) and  (4.3*\x,2.3*\y) ..  (4.4*\x,2.5*\y) .. controls (4.5*\x,2.9*\y) and (9.9*\x,2.9*\y) .. (n10);
        \draw[thick](n6) .. controls (6.1*\x,2.3*\y) and  (6.5*\x,2.3*\y) ..  (6.5*\x,2.5*\y) .. controls (6.6*\x,2.7*\y) and (7.9*\x,2.7*\y) .. (n8);
        \begin{pgfonlayer}{background}
          \fill[gray!50!white](1*\x,2.5*\y) circle(.2);
          \fill[gray!50!white](4*\x,2.5*\y) circle(.2);
          \fill[gray!50!white](5*\x,2.5*\y) circle(.2);
          \fill[gray!50!white](6*\x,2.5*\y) circle(.2);
          \fill[gray!50!white](7*\x,2.5*\y) circle(.2);
          \fill[gray!50!white](8*\x,2.5*\y) circle(.2);
          \fill[gray!50!white](9*\x,2.5*\y) circle(.2);
          \fill[gray!50!white](10*\x,2.5*\y) circle(.2);
          \fill[gray!50!white](1*\x,2.3*\y) -- (4*\x,2.3*\y) -- (4*\x,2.7*\y) -- (1*\x,2.7*\y) -- (1*\x,2.3*\y);
          \fill[gray!50!white](5*\x,2.3*\y) -- (6*\x,2.3*\y) -- (6*\x,2.7*\y) -- (5*\x,2.7*\y) -- (5*\x,2.3*\y);
          \fill[gray!50!white](8*\x,2.3*\y) -- (9*\x,2.3*\y) -- (9*\x,2.7*\y) -- (8*\x,2.7*\y) -- (8*\x,2.3*\y);
        \end{pgfonlayer}
        \draw(1*\x,1.5*\y) node{$\downarrow$};
        \draw(1*\x,1*\y) node{$1$};
        \draw(2*\x,1.5*\y) node{$\downarrow$};
        \draw(2*\x,1*\y) node{$7$};
        \draw(3*\x,1.5*\y) node{$\downarrow$};
        \draw(3*\x,1*\y) node{$9$};
        \draw(4*\x,1.5*\y) node{$\downarrow$};
        \draw(4*\x,1*\y) node{$10$};
        \draw(4.5*\x,1*\y) node{$\vert$};
        \draw(5*\x,1.5*\y) node{$\downarrow$};
        \draw(5*\x,1*\y) node{$2$};
        \draw(6*\x,1.5*\y) node{$\downarrow$};
        \draw(6*\x,1*\y) node{$5$};
        \draw(6.5*\x,1*\y) node{$\vert$};
        \draw(7*\x,1.5*\y) node{$\downarrow$};
        \draw(7*\x,1*\y) node{$3$};
        \draw(7.5*\x,1*\y) node{$\vert$};
        \draw(8*\x,1.5*\y) node{$\downarrow$};
        \draw(8*\x,1*\y) node{$4$};
        \draw(9*\x,1.5*\y) node{$\downarrow$};
        \draw(9*\x,1*\y) node{$6$};
        \draw(9.5*\x,1*\y) node{$\vert$};
        \draw(10*\x,1.5*\y) node{$\downarrow$};
        \draw(10*\x,1*\y) node{$8$};
      \end{tikzpicture}};
    \draw(7,1) node{
      \begin{tikzpicture}\small
        \def\x{1};
        \def\y{1};
        \draw(1*\x,1*\y) node(o1){$\{1\}$};
        \draw(2*\x,1*\y) node(o2){$\{2,9\}$};
        \draw(3.2*\x,1*\y) node(o3){$\{3,10\}$};
        \draw(4.2*\x,1*\y) node(o4){$\{4\}$};
        \draw(2*\x,2*\y) node(o5){$\{5\}$};
        \draw(3.2*\x,2*\y) node(o6){$\{6,8\}$};
        \draw(3.2*\x,3*\y) node(o7){$\{7\}$};
        \draw(o2) -- (o5);
        \draw(o2) -- (o6);
        \draw(o3) -- (o5);
        \draw(o3) -- (o6);
        \draw(o6) -- (o7);
      \end{tikzpicture}};
  \end{tikzpicture}
  \caption{The noncrossing partition $\Pf=\bigl\{\{1\},\{2,9\},\{3,10\},\{4\},\{5\},\{6,8\},\{7\}\bigr\}$ with respect to $n=10$ and $J=\{s_1,s_2,s_3,s_5,s_8\}$, the corresponding partially ordered set $\vec{O}_\Pf$, and the $(J,231)$-avoiding permutation constructed from $\Pf$.}
  \label{fig:bump_poset_illustration}
\end{figure}

\Cref{fig:bump_poset_illustration} illustrates this construction in the case $n=10$ and $J=\{s_1,s_2,s_3,s_5,s_8\}$.  We may now prove the following theorem.

\begin{theorem}\label{thm:j_sortables_noncrossings}
  For $n>0$ and $J\subseteq S$, there is a bijection $\mathfrak{S}_{n}^{J}(231) \simeq \NC^J_{n}$.
\end{theorem}
\begin{proof}
  Let $w\in\mathfrak{S}_{n}^{J}(231)$. We construct a set partition $\Pf$ of $[n]$ by associating a bump $(i,j)$ with every descent $(i,j)\in\des(w)$.  If $(i,j)$ is a bump of $\Pf$, then $(i,j)\in\des(w)$, and \Cref{lem:increasing_regions} implies that $i$ and $j$ lie in different $J$-regions.  This establishes condition \eqref{it:nc1}.  Suppose $(i_{1},i_{2})$ and $(j_{1},j_{2})$ are two different bumps of $\Pf$ with $i_{1}<j_{1}<i_{2}<j_{2}$, but neither $i_{1},j_{1}$ nor $i_{2},j_{1}$ are in the same $J$-region.  If $w_{i_{1}}<w_{j_{1}}$, then $(i_{1},j_{1},i_{2})$ is a ($J$,$231$)-pattern in $w$, which is a contradiction.  If $w_{i_{1}}>w_{j_{1}}$, it follows that $w_{j_{1}}<w_{i_{2}}$, and then $(j_{1},i_{2},j_{2})$ is a ($J$,$231$)-pattern in $w$, which is a contradiction.  Hence \eqref{it:nc2} is satisfied. Finally, suppose that $(i_{1},i_{2})$ and $(j_{1},j_{2})$ are two different bumps of $\Pf$ with $i_{1}<j_{1}<j_{2}<i_{2}$ such that $i_{1}$ and $j_{1}$ are in the same $J$-region.  \Cref{lem:increasing_regions} implies $w_{i_{1}}<w_{j_{1}}$.  It follows that $(i_{1},j_{1},i_{2})$ is a ($J$,$231$)-pattern in $w$, which is a contradiction.  Hence \eqref{it:nc3} is satisfied, and so $\Pf\in\NC^J_n$.

  \smallskip

  Conversely, let $\Pf\in\NC^J_n$.  We construct a permutation $w\in\mathfrak{S}_{n}^{J}(231)$ where every bump $(i,j)$ of $\Pf$ corresponds to a descent $(i,j)\in\des(w)$.  We proceed by induction on $n$, with the case $n=1$ being trivial.  Suppose that for any $n'<n$ we can construct a ($J'$,$231$)-avoiding permutation of $\mathfrak{S}_{n'}^{J'}$ from a given $J'$-noncrossing set partition of $[n']$, where $J'$ is the restriction of $J$ to $[n']$.  

Let $\vec{O}_{\Pf}$ be the partially ordered set on the parts of $\Pf$ that we have defined just before this theorem.  Let $\bar{P}$ be the unique part of $\Pf$ containing $1$, and let $X$ be the set of all integers that belong to parts in the order filter generated by $\bar{P}$.  We set $w(1)=\lvert X\rvert$.  

If we remove $\bar{P}$ from $\Pf$, then we obtain two smaller partitions from the remaining parts.  The elements in $X\setminus\bar{P}$ form a left partition $\Pf_l$, and the elements in $[n]\setminus X$ form a right partition $\Pf_r$.  Both $\Pf_{l}$ and $\Pf_{r}$ can be seen as parabolic noncrossing set partitions of $[n_l]$ and $[n_r]$, respectively, where $n_{l},n_{r}<n$.  By induction we can create $(J,231)$-avoiding permutations $w^{(l)}$ and $w^{(r)}$ from these partitions (and we may reuse $\vec{O}_{\Pf}$ for that). Now we obtain the value $w_{j}$ for $j\notin\bar{P}$ as follows.
  If $j\in\Pf_{l}$, then $w_{j}=w_{j'}^{(l)}$ if $j$ is the $(j')^{\text{th}}$ largest value in $\Pf_{l}$.
  If $j\in\Pf_{r}$, then $w_{j}=w_{j'}^{(r)}+\lvert X\rvert$ if $j$ is the $(j')^{\text{th}}$ largest value in $\Pf_{r}$.

  Since all bumps in $\Pf$ occur only between elements in $\bar{P}$, in $\Pf_{l}$, or in $\Pf_{r}$, it follows that $w\in\mathfrak{S}_{n}^{J}(231)$.
\end{proof}

\begin{example}
  Let $J=\{s_{1},s_{2},s_{3},s_{5},s_{8}\}$.  Consider $\Pf\in\NC_{10}^{J}$ given by the bumps $(2,9),(3,10),(6,8)$.  This partition is displayed in the top-left part of \Cref{fig:bump_poset_illustration}, the corresponding poset $\vec{O}_{\Pf}$ on the right.  Since no bump starts in $1$, we obtain $w_1=1$, and the corresponding right partition is the restriction of $\Pf$ to $\{2,3,\ldots,10\}$. Here we have $\bar{P}=\{2,9\}$, and we have $X=\{2,5,6,7,8,9\}$.  Hence we obtain $w_{2}=7$ and $w_{9}=6$. The corresponding left partition is $\Pf_{l}=\bigl\{\{5\},\{6,8\},\{7\}\bigr\}$ and the corresponding right partition is $\Pf_{r}=\bigl\{\{3,10\},\{4\}\bigr\}$.  By induction, we conclude that $w^{(l)}=1\;4\;|\;2\;|\;3$ and $w^{(r)}=2\;3\;|\;1$. We fashion them together to form the permutation $w=1\;7\;9\;10\;|\;2\;5\;|\;3\;|\;4\;6\;|\;8$, which is indeed contained in $\mathfrak{S}_{10}^{J}(231)$.  By construction, $\bigl\{(2,9),(3,10),(6,8)\bigr\}$ are the descents of $w$, and are precisely the bumps of $\Pf$.
\end{example}

\begin{remark}
  When restricted to the ($J$,$231$)-sortable elements, one can check that the bijection of \Cref{thm:j_sortables_noncrossings} is identical to the bijection given in~\cite{reading15noncrossing} between elements of the symmetric group and certain noncrossing arc diagrams. \Cref{thm:j_sortables_noncrossings} was discovered independently, and appeared in~\cite{williams13cataland}.
\end{remark}

\section{Parabolic Nonnesting Partitions}
  \label{sec:nonnesting}

The nonnesting set partitions are a second important subset of the set partitions of $[n]$.  \defn{Nonnesting} set partitions are characterized as not containing two bumps $(i_{1},i_{2})$ and $(j_{1},j_{2})$ such that $i_{1}<j_{1}<j_{2}<i_{2}$.  These were introduced by A.~Postnikov uniformly for all crystallographic Coxeter groups as order ideals in the corresponding root poset \cite{reiner97non}*{Remark~2}.  It turns out that (for any crystallographic Coxeter group) noncrossing and nonnesting partitions are equinumerous.  Moreover, they are also equidistributed by part size~\cite{athanasiadis98on}.  We introduce the following generalization.

\begin{definition}\label{def:parabolic_nonnesting_set_partitions}
  A partition $\Pf$ of $[n]$ is \defn{$J$-nonnesting} if it satisfies the following two conditions.
  \begin{description}
    \item[NN1\label{it:nn1}] If $i$ and $j$ lie in the same $J$-region, then they are not contained in the same part of $\Pf$.
    \item[NN2\label{it:nn2}] If $(i_{1},i_{2})$ and $(j_{1},j_{2})$ are two distinct bumps of $\Pf$, then it is not the case that $i_{1}<j_{1}<j_{2}<i_{2}$.
  \end{description}
\end{definition}

We denote the set of all $J$-nonnesting partitions of $[n]$ by $\NN^J_n$.  If $J=\emptyset$, then we recover the classical nonnesting set partitions.

Recall that the \defn{root poset} of $\mathfrak{S}_{n}$ is the poset $\Phi_{+}=(T,\leq)$, where $T$ is the set of all transpositions of $\mathfrak{S}_{n}$, and we have $(i_{1},i_{2})\leq(j_{1},j_{2})$ if and only if $i_{1}\geq j_{1}$ and $i_{2}\leq j_{2}$. The \defn{parabolic root poset} of $\mathfrak{S}_{n}$, denoted by $\Phi_{+}^{J}$, is the order filter of $\Phi_{+}$ generated by the adjacent transpositions \emph{not} in $J$.

We first observe that $J$-nonnesting partitions of $[n]$ are in bijection with order ideals in this parabolic root poset.  See Figure~\ref{fig:nonnesting_ideals} for an illustration.

\begin{lemma}
  For $n>0$ and $J\subseteq S$, there is a bijection from $J$-nonnesting partitions to order ideals in $\Phi_{+}^J$.
\end{lemma}
\begin{proof}
  Let $I$ be an order ideal of $\Phi_{+}^J$, and let $M$ denote the set of minimal elements in the complement of $I$.  In particular, $M$ is an antichain, \ie no two elements of $M$ are comparable.  Thus if $(i_1,i_2),(j_1,j_2)\in M$ and $i_1<j_1$, then $i_2\in\{j_1,j_1+1,\ldots,j_2-1\}$ (and accordingly if $j_1<i_1$).  Hence, \eqref{it:nn2} is satisfied.  If there are two distinct elements $i_0$ and $i_k$ which belong to the same $J$-region and to the same part $B$ of $\Pf$, then there must be a sequence of bumps $(i_0,i_1)$, $(i_1,i_2)$, $\ldots$, $(i_{k-1},i_k)$, which belong to this $J$-region and to $B$ as well.  This contradicts the definition of $\Phi_{+}^J$, because we have specifically excluded pairs of the form $(a,b)$ with $a$ and $b$ both belonging to the same $J$-region.  This contradiction shows that \eqref{it:nn1} is satisfied.

  Conversely, let $\Pf\in\NN_n^J$, and let $M$ be the set of bumps of $\Pf$.  By \eqref{it:nn1} we see that $M\subseteq\Phi_{+}^{J}$.  If there exist $(i_1,i_2),(j_1,j_2)\in M$ which are comparable in $\Phi_{+}^J$, then without loss of generality we may assume that $i_1\geq j_1$ and $i_2\leq j_2$.  Since naturally $i_1<i_2$ we obtain a contradiction to \eqref{it:nn2}.
\end{proof}

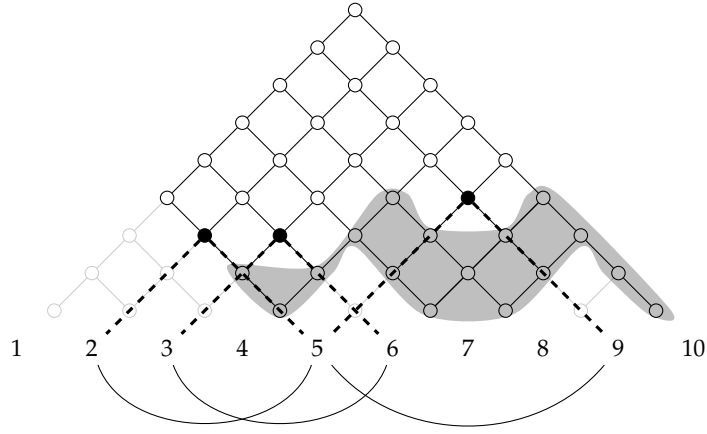
\begin{figure}
  \centering
  \begin{tikzpicture}\small
    \def\x{1};
    \def\y{.5};
    \draw[opacity=.2](1*\x,1*\y) node[draw,circle,scale=.6](p12){};
    \draw[opacity=.2](2*\x,1*\y) node[draw,circle,scale=.6](p23){};
    \draw[opacity=.2](3*\x,1*\y) node[draw,circle,scale=.6](p34){};
    \draw(4*\x,1*\y) node[draw,circle,scale=.6](p45){};
    \draw[opacity=.2](5*\x,1*\y) node[draw,circle,scale=.6](p56){};
    \draw(6*\x,1*\y) node[draw,circle,scale=.6](p67){};
    \draw(7*\x,1*\y) node[draw,circle,scale=.6](p78){};
    \draw[opacity=.2](8*\x,1*\y) node[draw,circle,scale=.6](p89){};
    \draw(9*\x,1*\y) node[draw,circle,scale=.6](p90){};
    \draw[opacity=.2](1.5*\x,2*\y) node[draw,circle,scale=.6](p13){};
    \draw[opacity=.2](2.5*\x,2*\y) node[draw,circle,scale=.6](p24){};
    \draw(3.5*\x,2*\y) node[draw,circle,scale=.6](p35){};
    \draw(4.5*\x,2*\y) node[draw,circle,scale=.6](p46){};
    \draw(5.5*\x,2*\y) node[draw,circle,scale=.6](p57){};
    \draw(6.5*\x,2*\y) node[draw,circle,scale=.6](p68){};
    \draw(7.5*\x,2*\y) node[draw,circle,scale=.6](p79){};
    \draw(8.5*\x,2*\y) node[draw,circle,scale=.6](p80){};
    \draw[opacity=.2](2*\x,3*\y) node[draw,circle,scale=.6](p14){};
    \draw(3*\x,3*\y) node[draw,fill,circle,scale=.6](p25){};
    \draw(4*\x,3*\y) node[draw,fill,circle,scale=.6](p36){};
    \draw(5*\x,3*\y) node[draw,circle,scale=.6](p47){};
    \draw(6*\x,3*\y) node[draw,circle,scale=.6](p58){};
    \draw(7*\x,3*\y) node[draw,circle,scale=.6](p69){};
    \draw(8*\x,3*\y) node[draw,circle,scale=.6](p70){};
    \draw(2.5*\x,4*\y) node[draw,circle,scale=.6](p15){};
    \draw(3.5*\x,4*\y) node[draw,circle,scale=.6](p26){};
    \draw(4.5*\x,4*\y) node[draw,circle,scale=.6](p37){};
    \draw(5.5*\x,4*\y) node[draw,circle,scale=.6](p48){};
    \draw(6.5*\x,4*\y) node[draw,fill,circle,scale=.6](p59){};
    \draw(7.5*\x,4*\y) node[draw,circle,scale=.6](p60){};
    \draw(3*\x,5*\y) node[draw,circle,scale=.6](p16){};
    \draw(4*\x,5*\y) node[draw,circle,scale=.6](p27){};
    \draw(5*\x,5*\y) node[draw,circle,scale=.6](p38){};
    \draw(6*\x,5*\y) node[draw,circle,scale=.6](p49){};
    \draw(7*\x,5*\y) node[draw,circle,scale=.6](p50){};
    \draw(3.5*\x,6*\y) node[draw,circle,scale=.6](p17){};
    \draw(4.5*\x,6*\y) node[draw,circle,scale=.6](p28){};
    \draw(5.5*\x,6*\y) node[draw,circle,scale=.6](p39){};
    \draw(6.5*\x,6*\y) node[draw,circle,scale=.6](p40){};
    \draw(4*\x,7*\y) node[draw,circle,scale=.6](p18){};
    \draw(5*\x,7*\y) node[draw,circle,scale=.6](p29){};
    \draw(6*\x,7*\y) node[draw,circle,scale=.6](p30){};
    \draw(4.5*\x,8*\y) node[draw,circle,scale=.6](p19){};
    \draw(5.5*\x,8*\y) node[draw,circle,scale=.6](p20){};
    \draw(5*\x,9*\y) node[draw,circle,scale=.6](p10){};
    \draw[opacity=.2](p12) -- (p13) -- (p14) -- (p15);
    \draw(p15) -- (p16) -- (p17) -- (p18) -- (p19) -- (p10);
    \draw[opacity=.2](p23) -- (p24) -- (p25);
    \draw(p25) -- (p26) -- (p27) -- (p28) -- (p29) -- (p20);
    \draw[opacity=.2](p34) -- (p35);
    \draw(p35)-- (p36) -- (p37) -- (p38) -- (p39) -- (p30);
    \draw(p45) -- (p46) -- (p47) -- (p48) -- (p49) -- (p40);
    \draw[opacity=.2](p56) -- (p57);
    \draw(p57)-- (p58) -- (p59) -- (p50);
    \draw(p67) -- (p68) -- (p69) -- (p60);
    \draw(p78) -- (p79) -- (p70);
    \draw[opacity=.2](p89) -- (p80);
    \draw(p67) -- (p68) -- (p69) -- (p60);
    \draw[opacity=.2](p23) -- (p13);
    \draw[opacity=.2](p34) -- (p24) -- (p14);
    \draw(p45) -- (p35) -- (p25) -- (p15);
    \draw[opacity=.2](p56) -- (p46);
    \draw(p46)-- (p36) -- (p26) -- (p16);
    \draw(p67) -- (p57) -- (p47) -- (p37) -- (p27) -- (p17);
    \draw(p78) -- (p68) -- (p58) -- (p48) -- (p38) -- (p28) -- (p18);
    \draw[opacity=.2](p89) -- (p79);
    \draw(p79) -- (p69) -- (p59) -- (p49) -- (p39) -- (p29) -- (p19);
    \draw(p90) -- (p80) -- (p70) -- (p60) -- (p50) -- (p40) -- (p30) -- (p20) -- (p10);
    \begin{pgfonlayer}{background}
      \fill[gray!50!white] plot[smooth cycle] coordinates{(4*\x,.8*\y) (3.3*\x,2.2*\y) (4.5*\x,2.2*\y) (5.3*\x,4.1*\y) (5.7*\x,4.1*\y) (6*\x,3.2*\y) (7*\x,3.2*\y) (7.3*\x,4.1*\y) (7.7*\x,4.1*\y) (9.2*\x,1*\y) (8.9*\x,.9*\y) (8.3*\x,2*\y) (8*\x,2.7*\y) (7.1*\x,.9*\y) (5.9*\x,.9*\y) (5*\x,2.7*\y) (4.7*\x,2*\y) };
    \end{pgfonlayer}
    \draw(.5*\x,0*\y) node(n1){$1$};
    \draw(1.5*\x,0*\y) node(n2){$2$};
    \draw(2.5*\x,0*\y) node(n3){$3$};
    \draw(3.5*\x,0*\y) node(n4){$4$};
    \draw(4.5*\x,0*\y) node(n5){$5$};
    \draw(5.5*\x,0*\y) node(n6){$6$};
    \draw(6.5*\x,0*\y) node(n7){$7$};
    \draw(7.5*\x,0*\y) node(n8){$8$};
    \draw(8.5*\x,0*\y) node(n9){$9$};
    \draw(9.5*\x,0*\y) node(n0){$10$};
    \draw[dashed,very thick](p25) -- (n2);
    \draw[dashed,very thick](p25) -- (n5);
    \draw[dashed,very thick](p36) -- (n3);
    \draw[dashed,very thick](p36) -- (n6);
    \draw[dashed,very thick](p59) -- (n5);
    \draw[dashed,very thick](p59) -- (n9);
    \draw(n2) to[bend right=70](n5);
    \draw(n3) to[bend right=70](n6);
    \draw(n5) to[bend right=50](n9);
  \end{tikzpicture}
  \caption{The parabolic root poset of $\mathfrak{S}_{10}$ with respect to $J=\{s_1,s_2,s_3,s_5,s_8\}$.  The shaded region is an order ideal, and the minimal elements of the complement are marked in black.  The corresponding $J$-nonnesting partition is $\bigl\{\{1\},\{2,5,9\},\{3,6\},\{4\},\{7\},\{8\},\{10\}\bigr\}$.}
  \label{fig:nonnesting_ideals}
\end{figure}

We now prove that $J$-nonnesting and $J$-noncrossing partitions are also in bijection.  See \Cref{fig:j_nonnesting_example} for an example.

\begin{theorem}\label{thm:j_noncrossings_nonnestings}
  For $n>0$ and $J\subseteq S$, there is a bijection $\NN^J_n\simeq\NC^J_n$.
\end{theorem}
\begin{proof}
  We begin with the construction of a bijection from $\NN^J_n$ to $\NC^J_n$ for the case of maximal parabolic quotients, \ie where $J=S\setminus\{s_{k}\}$ for $k\in[n]$. We label the transposition $(i,j)$ in $\Phi_{+}^{J}$ by the arc $(k+1-i,n+1-j+k)$, which yields the following labeling of $\Phi_{+}^{J}$ (under a suitable rotation):
  \begin{displaymath}\begin{array}{cccc}
    \bigl(k,(k+1)\bigr) & \cdots & \bigl(k,(n-1)\bigr) & \bigl(k,n\bigr)\\
    \vdots & \vdots & \vdots & \vdots\\
    \bigl(2,(k+1)\bigr) & \cdots & \bigl(2,(n-1)\bigr) & \bigl(2,n\bigr)\\
    \bigl(1,(k+1)\bigr) & \cdots & \bigl(1,(n-1)\bigr) & \bigl(1,n\bigr)\\
  \end{array}\end{displaymath}
  The $J$-nonnesting set partition corresponding to an order ideal in $\Phi_{+}^{J}$ is the one whose bumps are the labels of the minimal elements not in the order ideal. Since $B(J)=\bigl\{\{1,2,\ldots,k\},\{k+1,k+2,\ldots,n\}\bigr\}$, condition \eqref{it:nc3} ensures that every $J$-noncrossing partition is also $J$-nonnesting and vice versa.

  Now suppose that $J=S\setminus\{s_{k_{1}},s_{k_{2}},\ldots,s_{k_{r}}\}$, and let $I$ be an order ideal of $\Phi_{+}^{J}$.  We construct a noncrossing partition $\Pf\in\NC^J_n$ inductively starting from the partition with no parts.  First we break $I$ in two pieces, $A$ and $B$: $A$ contains all the transpositions in $I$ that lie above $s_{k_{1}}$ in $\Phi_{+}^{J}$, and $B$ contains all the other transpositions in $I$.  Then $B$ is an order ideal in $\Phi_{+}^{J\setminus\{s_{k_{1}}\}}$, and we can construct a $\bigl(J\setminus\{s_{k_{1}}\}\bigr)$-noncrossing set partition of $\{k_{1}+1,k_{1}+2,\ldots,n\}$ by induction.  Now we choose all those columns in piece $A$ that either lie outside the order filter generated by $s_{k_{2}},\ldots,s_{k_{r}}$ or that have an element of $I$ in piece $B$ directly below them. (We thus pick the columns of $A$ that are ``supported'' by $B$.)

  Let $l_{1},l_{2},\ldots,l_{r}$ denote the column labels from the inductive step of the part of $B$ that supports $A$.  Any bump starting in $\{1,2,\ldots,k_{1}\}$ can end either in $\{k_{1}+1,k_{1}+2,\ldots,k_{2}\}$ or in $\{l_{1},l_{2},\ldots,l_{r}\}$, in order not to cross any existing bumps.  We label the transpositions in the chosen part of $A$ as follows:
  \begin{displaymath}\begin{array}{cccccccc}
    \bigl(k_{1},(k_{1}+1)\bigr) & \bigl(k_{1},(k_{1}+2)\bigr) & \cdots & \bigl(k_{1},k_{2}\bigr) & \bigl(k,l_{1}\bigr) & \bigl(k,l_{2}\bigr) & \cdots & \bigl(k,l_{r}\bigr)\\
    \vdots & \vdots & \vdots & \vdots & \vdots & \vdots & \vdots & \vdots\\
    \bigl(2,(k_{1}+1)\bigr) & \bigl(2,(k_{1}+2)\bigr) & \cdots & \bigl(2,k_{2}\bigr) & \bigl(2,l_{1}\bigr) & \bigl(2,l_{2}\bigr) & \cdots & \bigl(2,l_{r}\bigr)\\
    \bigl(1,(k_{1}+1)\bigr) & \bigl(1,(k_{1}+2)\bigr) & \cdots & \bigl(1,k_{2}\bigr) & \bigl(1,l_{1}\bigr) & \bigl(1,l_{2}\bigr) & \cdots & \bigl(1,l_{r}\bigr)\\
  \end{array}\end{displaymath}
  The labels corresponding to the minimal transpositions that are not in $I$ within these chosen columns then yield the remaining bumps.  By construction, the resulting partition is $J$-noncrossing.

  The inverse map is constructed by temporarily forgetting about the bumps from the first $J$-region, and then using the smaller $J$-noncrossing partition to construct the $B$ piece of the order ideal inductively.  From there, we can again identify the ``supported'' columns in the $A$ piece, and the bumps starting in the first $J$-region then give the remaining elements of the order ideal.
\end{proof}

\begin{example}
  Consider the $J$-nonnesting set partition of $[10]$ shown at the top left of \Cref{fig:j_nonnesting_example} indicated by the dark gray region. The construction of the smaller parabolic noncrossing partitions is shown in the middle and right part of that figure, and the resulting $J$-noncrossing partition is shown at the bottom left.
\end{example}

\begin{figure}
  \centering
  \begin{tikzpicture}
    \draw(-1.15,.5) -- (11.6,.5) -- (11.6,6) -- (-1.15,6) -- (-1.15,.5);
    \draw(3.1,.5) -- (3.1,6);
    \draw(7.35,.5) -- (7.35,6);
    \draw(1,4) node{\begin{tikzpicture}\small
        \def\x{.4};
        \def\y{.4};
        \draw(1*\x,1*\y) -- (2*\x,1*\y) -- (2*\x,3*\y) -- (4*\x,3*\y) -- (4*\x,4*\y) -- (5*\x,4*\y) -- (5*\x,6*\y) -- (7*\x,6*\y) -- (7*\x,10*\y) -- (3*\x,10*\y) -- (3*\x,6*\y) -- (1*\x,6*\y) -- (1*\x,1*\y);
        \draw(1*\x,5*\y) -- (3*\x,5*\y) -- (3*\x,6*\y);
        \draw(3*\x,7*\y) -- (6*\x,7*\y) -- (6*\x,8*\y) -- (7*\x,8*\y);
        \draw[dotted](3*\x,6*\y) -- (5*\x,6*\y);
        \draw(5*\x,9*\y) node{Piece A};
        \draw(2.5*\x,4*\y) node{Piece B};
        \draw(6.5*\x,6.5*\y) node[scale=.5]{$(4,5)$};
        \draw(2.5*\x,5.5*\y) node{$\bullet$};
        \draw(5.5*\x,7.5*\y) node{$\bullet$};
        \draw(6.5*\x,8.5*\y) node{$\bullet$};
        \begin{pgfonlayer}{background}
          \draw[gray!50!white](2*\x,2*\y) -- (3*\x,2*\y) -- (3*\x,3*\y);
          \draw[gray!50!white](5*\x,5*\y) -- (6*\x,5*\y) -- (6*\x,6*\y);
          \draw[gray!50!white](7*\x,7*\y) -- (8*\x,7*\y) -- (8*\x,8*\y) -- (9*\x,8*\y) -- (9*\x,9*\y) -- (10*\x,9*\y) -- (10*\x,10*\y) -- (7*\x,10*\y);
          \draw[gray!50!white](3*\x,10*\y) -- (1*\x,10*\y) -- (1*\x,6*\y);
          \fill[gray!75!white](1*\x,1*\y) -- (2*\x,1*\y) -- (2*\x,3*\y) -- (4*\x,3*\y) -- (4*\x,4*\y) -- (5*\x,4*\y) -- (5*\x,6*\y) -- (7*\x,6*\y) -- (7*\x,8*\y) -- (6*\x,8*\y) -- (6*\x,7*\y) -- (3*\x,7*\y) -- (3*\x,5*\y) -- (1*\x,5*\y) -- (1*\x,1*\y);
          \fill[gray!25!white](1*\x,5*\y) -- (3*\x,5*\y) -- (3*\x,6*\y) -- (1*\x,6*\y) -- (1*\x,5*\y);
          \fill[gray!25!white](3*\x,7*\y) -- (6*\x,7*\y) -- (6*\x,8*\y) -- (7*\x,8*\y) -- (7*\x,10*\y) -- (3*\x,10*\y) -- (3*\x,7*\y);
        \end{pgfonlayer}
      \end{tikzpicture}};
    \draw(1,1) node{\begin{tikzpicture}\tiny
        \def\x{.4};
        \def\y{1};
        \draw(1*\x,1*\y) node[circle,fill,scale=.33](n1){};
        \draw(1*\x,.7*\y) node{$1$};
        \draw(2*\x,1*\y) node[circle,fill,scale=.33](n2){};
        \draw(2*\x,.7*\y) node{$2$};
        \draw(3*\x,1*\y) node[circle,fill,scale=.33](n3){};
        \draw(3*\x,.7*\y) node{$3$};
        \draw(4*\x,1*\y) node[circle,fill,scale=.33](n4){};
        \draw(4*\x,.7*\y) node{$4$};
        \draw(5*\x,1*\y) node[circle,fill,scale=.33](n5){};
        \draw(5*\x,.7*\y) node{$5$};
        \draw(6*\x,1*\y) node[circle,fill,scale=.33](n6){};
        \draw(6*\x,.7*\y) node{$6$};
        \draw(7*\x,1*\y) node[circle,fill,scale=.33](n7){};
        \draw(7*\x,.7*\y) node{$7$};
        \draw(8*\x,1*\y) node[circle,fill,scale=.33](n8){};
        \draw(8*\x,.7*\y) node{$8$};
        \draw(9*\x,1*\y) node[circle,fill,scale=.33](n9){};
        \draw(9*\x,.7*\y) node{$9$};
        \draw(10*\x,1*\y) node[circle,fill,scale=.33](n10){};
        \draw(10*\x,.7*\y) node{$10$};
        \draw(n2) .. controls (2.1*\x,.7*\y) and  (4.6*\x,.7*\y) ..  (4.7*\x,1*\y) .. controls (4.8*\x,1.3*\y) and (8.9*\x,1.3*\y) .. (n9);
        \draw(n3) .. controls (3.1*\x,.8*\y) and  (4.3*\x,.8*\y) ..  (4.4*\x,1*\y) .. controls (4.5*\x,1.4*\y) and (9.9*\x,1.4*\y) .. (n10);
        \draw(n6) .. controls (6.1*\x,.8*\y) and  (6.5*\x,.8*\y) ..  (6.5*\x,1*\y) .. controls (6.6*\x,1.2*\y) and (7.9*\x,1.2*\y) .. (n8);
        \begin{pgfonlayer}{background}
          \fill[gray!50!white](1*\x,1*\y) circle(.1);
          \fill[gray!50!white](4*\x,1*\y) circle(.1);
          \fill[gray!50!white](5*\x,1*\y) circle(.1);
          \fill[gray!50!white](6*\x,1*\y) circle(.1);
          \fill[gray!50!white](7*\x,1*\y) circle(.1);
          \fill[gray!50!white](8*\x,1*\y) circle(.1);
          \fill[gray!50!white](9*\x,1*\y) circle(.1);
          \fill[gray!50!white](10*\x,1*\y) circle(.1);
          \fill[gray!50!white](1*\x,.9*\y) -- (4*\x,.9*\y) -- (4*\x,1.1*\y) -- (1*\x,1.1*\y) -- (1*\x,.9*\y);
          \fill[gray!50!white](5*\x,.9*\y) -- (6*\x,.9*\y) -- (6*\x,1.1*\y) -- (5*\x,1.1*\y) -- (5*\x,.9*\y);
          \fill[gray!50!white](8*\x,.9*\y) -- (9*\x,.9*\y) -- (9*\x,1.1*\y) -- (8*\x,1.1*\y) -- (8*\x,.9*\y);
        \end{pgfonlayer}
      \end{tikzpicture}};
    \draw(5.25,4) node{\begin{tikzpicture}\small
        \def\x{.4};
        \def\y{.4};
        \draw(3*\x,6*\y) -- (7*\x,6*\y) -- (7*\x,10*\y) -- (3*\x,10*\y) -- (3*\x,6*\y);
        \draw(3*\x,7*\y) -- (6*\x,7*\y) -- (6*\x,8*\y) -- (7*\x,8*\y);
        \draw[dotted](3*\x,6*\y) -- (5*\x,6*\y);
        \draw(5*\x,9*\y) node{Piece A};
        \draw(6.5*\x,6.5*\y) node[scale=.5]{$(4,5)$};
        \draw(5.5*\x,7.5*\y) node{$\bullet$};
        \draw(6.5*\x,8.5*\y) node{$\bullet$};
        \begin{pgfonlayer}{background}
          \draw[gray!50!white](1*\x,1*\y) -- (2*\x,1*\y) -- (2*\x,2*\y) -- (3*\x,2*\y) -- (3*\x,3*\y) -- (4*\x,3*\y) -- (4*\x,4*\y) -- (5*\x,4*\y) -- (5*\x,5*\y) -- (6*\x,5*\y) -- (6*\x,6*\y);
          \draw[gray!50!white](7*\x,7*\y) -- (8*\x,7*\y) -- (8*\x,8*\y) -- (9*\x,8*\y) -- (9*\x,9*\y) -- (10*\x,9*\y) -- (10*\x,10*\y) -- (7*\x,10*\y);
          \draw[gray!50!white](3*\x,10*\y) -- (1*\x,10*\y) -- (1*\x,1*\y);
          \fill[gray!75!white](3*\x,6*\y) -- (7*\x,6*\y) -- (7*\x,8*\y) -- (6*\x,8*\y) -- (6*\x,7*\y) -- (3*\x,7*\y) -- (3*\x,6*\y);
          \fill[gray!25!white](3*\x,7*\y) -- (6*\x,7*\y) -- (6*\x,8*\y) -- (7*\x,8*\y) -- (7*\x,10*\y) -- (3*\x,10*\y) -- (3*\x,7*\y);
        \end{pgfonlayer}
      \end{tikzpicture}};
    \draw(5.25,1) node{\begin{tikzpicture}\tiny
        \def\x{.4};
        \def\y{1};
        \draw(1*\x,1*\y) node[circle,fill,scale=.33](n1){};
        \draw(1*\x,.7*\y) node{$1$};
        \draw(2*\x,1*\y) node[circle,fill,scale=.33](n2){};
        \draw(2*\x,.7*\y) node{$2$};
        \draw(3*\x,1*\y) node[circle,fill,scale=.33](n3){};
        \draw(3*\x,.7*\y) node{$3$};
        \draw(4*\x,1*\y) node[circle,fill,scale=.33](n4){};
        \draw(4*\x,.7*\y) node{$4$};
        \draw[gray!25!white](5*\x,1*\y) node[circle,fill,scale=.33](n5){};
        \draw[gray!25!white](5*\x,.7*\y) node{$5$};
        \draw[gray!25!white](6*\x,1*\y) node[circle,fill,scale=.33](n6){};
        \draw[gray!25!white](6*\x,.7*\y) node{$6$};
        \draw[gray!25!white](7*\x,1*\y) node[circle,fill,scale=.33](n7){};
        \draw[gray!25!white](7*\x,.7*\y) node{$7$};
        \draw[gray!25!white](8*\x,1*\y) node[circle,fill,scale=.33](n8){};
        \draw[gray!25!white](8*\x,.7*\y) node{$8$};
        \draw(9*\x,1*\y) node[circle,fill,scale=.33](n9){};
        \draw(9*\x,.7*\y) node{$9$};
        \draw(10*\x,1*\y) node[circle,fill,scale=.33](n10){};
        \draw(10*\x,.7*\y) node{$10$};
        \draw(n2) .. controls (2.1*\x,.7*\y) and  (4.6*\x,.7*\y) ..  (4.7*\x,1*\y) .. controls (4.8*\x,1.3*\y) and (8.9*\x,1.3*\y) .. (n9);
        \draw(n3) .. controls (3.1*\x,.8*\y) and  (4.3*\x,.8*\y) ..  (4.4*\x,1*\y) .. controls (4.5*\x,1.4*\y) and (9.9*\x,1.4*\y) .. (n10);
        \draw[gray!25!white](n6) .. controls (6.1*\x,.8*\y) and  (6.5*\x,.8*\y) ..  (6.5*\x,1*\y) .. controls (6.6*\x,1.2*\y) and (7.9*\x,1.2*\y) .. (n8);
        \begin{pgfonlayer}{background}
          \fill[gray!50!white](1*\x,1*\y) circle(.1);
          \fill[gray!50!white](4*\x,1*\y) circle(.1);
          \fill[gray!25!white](5*\x,1*\y) circle(.1);
          \fill[gray!25!white](6*\x,1*\y) circle(.1);
          \fill[gray!25!white](7*\x,1*\y) circle(.1);
          \fill[gray!25!white](8*\x,1*\y) circle(.1);
          \fill[gray!50!white](9*\x,1*\y) circle(.1);
          \fill[gray!50!white](10*\x,1*\y) circle(.1);
          \fill[gray!50!white](1*\x,.9*\y) -- (4*\x,.9*\y) -- (4*\x,1.1*\y) -- (1*\x,1.1*\y) -- (1*\x,.9*\y);
          \fill[gray!25!white](5*\x,.9*\y) -- (6*\x,.9*\y) -- (6*\x,1.1*\y) -- (5*\x,1.1*\y) -- (5*\x,.9*\y);
          \fill[gray!50!white](8*\x,.9*\y) -- (9*\x,.9*\y) -- (9*\x,1.1*\y) -- (8*\x,1.1*\y) -- (8*\x,.9*\y);
        \end{pgfonlayer}
        \begin{pgfonlayer}{middle}
          \fill[gray!25!white](8*\x,.9*\y) -- (8.5*\x,.9*\y) -- (8.5*\x,1.1*\y) -- (8*\x,1.1*\y) -- (8*\x,.9*\y);
        \end{pgfonlayer}
      \end{tikzpicture}};
    \draw(9.5,4) node{\begin{tikzpicture}\small
        \def\x{.4};
        \def\y{.4};
        \draw(1*\x,1*\y) -- (2*\x,1*\y) -- (2*\x,3*\y) -- (4*\x,3*\y) -- (4*\x,4*\y) -- (5*\x,4*\y) -- (5*\x,6*\y) -- (1*\x,6*\y) -- (1*\x,1*\y);
        \draw(1*\x,5*\y) -- (3*\x,5*\y) -- (3*\x,6*\y);
        \draw(2.5*\x,4*\y) node{Piece B};
        \draw(2.5*\x,5.5*\y) node{$\bullet$};
        \begin{pgfonlayer}{background}
          \draw[gray!50!white](2*\x,2*\y) -- (3*\x,2*\y) -- (3*\x,3*\y);
          \draw[gray!50!white](5*\x,5*\y) -- (6*\x,5*\y) -- (6*\x,6*\y) -- (7*\x,6*\y) -- (7*\y,7*\y) -- (8*\x,7*\y) -- (8*\x,8*\y) -- (9*\x,8*\y) -- (9*\x,9*\y) -- (10*\x,9*\y) -- (10*\x,10*\y) -- (1*\x,10*\y) -- (1*\x,6*\y);
          \fill[gray!75!white](1*\x,1*\y) -- (2*\x,1*\y) -- (2*\x,3*\y) -- (4*\x,3*\y) -- (4*\x,4*\y) -- (5*\x,4*\y) -- (5*\x,6*\y) -- (3*\x,6*\y) -- (3*\x,5*\y) -- (1*\x,5*\y) -- (1*\x,1*\y);
          \fill[gray!25!white](1*\x,5*\y) -- (3*\x,5*\y) -- (3*\x,6*\y) -- (1*\x,6*\y) -- (1*\x,5*\y);
        \end{pgfonlayer}
      \end{tikzpicture}};
    \draw(9.5,1) node{\begin{tikzpicture}\tiny
        \def\x{.4};
        \def\y{1};
        \draw[gray!25!white](1*\x,1*\y) node[circle,fill,scale=.33](n1){};
        \draw[gray!25!white](1*\x,.7*\y) node{$1$};
        \draw[gray!25!white](2*\x,1*\y) node[circle,fill,scale=.33](n2){};
        \draw[gray!25!white](2*\x,.7*\y) node{$2$};
        \draw[gray!25!white](3*\x,1*\y) node[circle,fill,scale=.33](n3){};
        \draw[gray!25!white](3*\x,.7*\y) node{$3$};
        \draw[gray!25!white](4*\x,1*\y) node[circle,fill,scale=.33](n4){};
        \draw[gray!25!white](4*\x,.7*\y) node{$4$};
        \draw(5*\x,1*\y) node[circle,fill,scale=.33](n5){};
        \draw(5*\x,.7*\y) node{$5$};
        \draw(6*\x,1*\y) node[circle,fill,scale=.33](n6){};
        \draw(6*\x,.7*\y) node{$6$};
        \draw(7*\x,1*\y) node[circle,fill,scale=.33](n7){};
        \draw(7*\x,.7*\y) node{$7$};
        \draw(8*\x,1*\y) node[circle,fill,scale=.33](n8){};
        \draw(8*\x,.7*\y) node{$8$};
        \draw(9*\x,1*\y) node[circle,fill,scale=.33](n9){};
        \draw(9*\x,.7*\y) node{$9$};
        \draw(10*\x,1*\y) node[circle,fill,scale=.33](n10){};
        \draw(10*\x,.7*\y) node{$10$};
        \draw[gray!25!white](n2) .. controls (2.1*\x,.7*\y) and  (4.6*\x,.7*\y) ..  (4.7*\x,1*\y) .. controls (4.8*\x,1.3*\y) and (8.9*\x,1.3*\y) .. (n9);
        \draw[gray!25!white](n3) .. controls (3.1*\x,.8*\y) and  (4.3*\x,.8*\y) ..  (4.4*\x,1*\y) .. controls (4.5*\x,1.4*\y) and (9.9*\x,1.4*\y) .. (n10);
        \draw(n6) .. controls (6.1*\x,.8*\y) and  (6.5*\x,.8*\y) ..  (6.5*\x,1*\y) .. controls (6.6*\x,1.2*\y) and (7.9*\x,1.2*\y) .. (n8);
        \begin{pgfonlayer}{background}
          \fill[gray!25!white](1*\x,1*\y) circle(.1);
          \fill[gray!25!white](4*\x,1*\y) circle(.1);
          \fill[gray!50!white](5*\x,1*\y) circle(.1);
          \fill[gray!50!white](6*\x,1*\y) circle(.1);
          \fill[gray!50!white](7*\x,1*\y) circle(.1);
          \fill[gray!50!white](8*\x,1*\y) circle(.1);
          \fill[gray!50!white](9*\x,1*\y) circle(.1);
          \fill[gray!50!white](10*\x,1*\y) circle(.1);
          \fill[gray!25!white](1*\x,.9*\y) -- (4*\x,.9*\y) -- (4*\x,1.1*\y) -- (1*\x,1.1*\y) -- (1*\x,.9*\y);
          \fill[gray!50!white](5*\x,.9*\y) -- (6*\x,.9*\y) -- (6*\x,1.1*\y) -- (5*\x,1.1*\y) -- (5*\x,.9*\y);
          \fill[gray!50!white](8*\x,.9*\y) -- (9*\x,.9*\y) -- (9*\x,1.1*\y) -- (8*\x,1.1*\y) -- (8*\x,.9*\y);
        \end{pgfonlayer}
      \end{tikzpicture}};
  \end{tikzpicture}
  \caption{Figure illustrating how to combine pieces A and B.}
  \label{fig:j_nonnesting_example}
\end{figure}

\begin{proof}[Proof of \Cref{thm:parabolic_bijections}]
    This follows from \Cref{thm:j_sortables_noncrossings,thm:j_noncrossings_nonnestings}.
\end{proof}

Recall that a \defn{Ferrers shape} is a sequence $\lambda=(\lambda_1,\lambda_2,\ldots,\lambda_k)$ of positive integers with $\lambda_1\geq\lambda_2\geq\cdots\geq\lambda_k$.  Another Ferrers shape $(\lambda'_1,\lambda'_2,\ldots,\lambda'_k)$ \defn{fits} inside $\lambda$ if $\lambda'_i\leq\lambda_i$ for all $i\in[k]$.

The parabolic root poset---when rotated by 45 degrees counterclockwise---can be viewed as a \emph{bounding} Ferrers shape 
\begin{displaymath}
  \lambda_n^J := (j_r^{n-j_r},\ldots,j_2^{j_3-j_2},j_1^{j_2-j_1})
\end{displaymath} 
and the complements of the order ideals in this poset correspond precisely to the Ferrers shapes that fit inside $\lambda_n^J$.  In our running example we obtain $\lambda_{10}^{\{s_1,s_2,s_3,s_5,s_8\}}=(9,7,7,6,4,4)$ as can quickly be verified in \Cref{fig:nonnesting_ideals}.  The enumeration of Ferrers shapes fitting into $\lambda_n^J$ now follows from the next result due to G.~Kreweras, and allows for computing the parabolic Catalan numbers.

\begin{theorem}[\cite{kreweras65sur}*{Section~2.3.7}]\label{}
  If $\lambda=(\lambda_1,\lambda_2,\ldots,\lambda_k)$ is a Ferrers shape, then the number of Ferrers shapes that fits inside $\lambda$ is given by the determinant of the $k\times k$-matrix whose entry in row $i$ and column $j$ is $\binom{\lambda_j+1}{j-i+1}$.
\end{theorem}

\section{Generalization to Coxeter Groups}
  \label{sec:generalizations}
In recent years, $231$-avoiding permutations, noncrossing set partitions, and nonnesting set partitions have each been generalized from the symmetric group to finite Coxeter groups---see \cite{reading07sortable}, \cites{bessis03dual,brady08non}, and \cite{reiner97non}*{Remark~2}, respectively.  These generalizations allow for further parametrization by a Coxeter element (a product of the simple reflections in some order).  In this section, we describe a generalization of our parabolic versions of these objects in a similar fashion.

\subsection{Coxeter Groups}
  \label{sec:coxeter_groups}
We first recall some background on Coxeter groups.  For more details see \cites{bjorner05combinatorics,humphreys90reflection}.  A \defn{Coxeter system} is a pair $(W,S)$, where $W$ is a group and $S=\{s_{1},s_{2},\ldots,s_{n}\}\subseteq W$ is a generating set such that $W$ admits the presentation
\begin{equation}\label{eq:coxeter_presentation}
  W = \bigl\langle s_{1},s_{2},\ldots,s_{n}\mid (s_{i}s_{j})^{m_{i,j}}=\id\;\text{for}\;i,j\in[n]\bigr\rangle,
\end{equation}
where $\id$ denotes the identity of $W$.  The parameters $m_{i,j}$ are positive integers or the formal symbol $\infty$, where $m_{i,j}=1$ if and only if $i=j$.  If $m_{i,j}=\infty$, then there is no relation between the generators $s_{i}$ and $s_{j}$.  In this situation we call $W$ a \defn{Coxeter group} and the cardinality $n=\lvert S\rvert$ the \defn{rank} of $W$.  For geometric reasons we call the elements of $S$ the \defn{simple reflections}, and define the set of all \defn{reflections} by $T=\{wsw^{-1}\mid w\in W,s\in S\}$.

Since $S$ generates $W$, every $w\in W$ can be written as a product of the elements in $S$.  A \defn{reduced expression} for $w$ is such a product of minimal length, and this length is called the \defn{Coxeter length} of $w$; denoted by $\ell_{S}(w)$.  The \defn{(right) weak order} on $W$ is the partial order $\leq_{S}$ defined by $u\leq_{S}v$ if and only if $\ell_{S}(v)=\ell_{S}(u)+\ell_{S}(u^{-1}v)$.  We write $\Weak(W)$ for the partially ordered set $(W,\leq_{S})$.

\begin{remark}\label{rem:left_right}
  In \Cref{sec:weak_order_symmetric_group} we defined a left weak order on the symmetric group $\mathfrak{S}_{n}$, which may be generalized to Coxeter groups via the condition $\ell_{S}(v)=\ell_{S}(u)+\ell_{S}(vu^{-1})$.  The map $w\mapsto w^{-1}$ is a poset isomorphism from left to right weak order, so that the results from \Cref{sec:tamari_pq} could be phrased equally well in terms of right weak order.
\end{remark}

In general, $\Weak(W)$ is a meet-semilattice---if $W$ is finite, then $\Weak(W)$ has a unique longest element $\wo$, which implies that $\Weak(W)$ is in fact a lattice~\cite{bjorner05combinatorics}*{Theorem~3.2.1}.

Any $J\subseteq S$ naturally generates a (parabolic) subgroup $W_{J}$ of $W$, and the set of minimal length representatives of the right cosets of $W$ by $W_{J}$ forms the \defn{parabolic quotient} $W^{J}$ of $W$ with respect to $J$.  Proposition~2.4.4 in \cite{bjorner05combinatorics} implies that any $w\in W$ can be factorized uniquely as $w=w^{J}\cdot w_{J}$, where $w^{J}\in W^{J}$ and $w_{J}\in W_{J}$.  The weak order on $W$ gives a partial order on $W^{J}$, and $\Weak(W^{J})$ is isomorphic to the weak order ideal $[\id,\wo^{J}]$ when $W$ is finite~\cite{bjorner88generalized}*{Theorem~4.1}.

Comparing the presentations \eqref{eq:symmetric_presentation} and \eqref{eq:coxeter_presentation}, we see that the symmetric group with the generating set of all adjacent transpositions forms a Coxeter system, and the reflections are all conjugates of the adjacent transpositions.  We use this correspondence to generalize the notion of an inversion from the symmetric group to all Coxeter groups.  A \defn{(left) inversion} of $W$ is a reflection $t\in T$ such that $\ell_{S}(tw)<\ell_{S}(w)$.  The set of all inversions of $w$ is denoted by $\inv(w)$.  Analogously to the symmetric group, we can give an equivalent definition of the weak order by setting $u\leq_{S}v$ if and only if $\inv(u)\subseteq\inv(v)$~\cite{bjorner05combinatorics}*{Proposition~3.1.3}.

We want to emphasize a special subset of the inversions.  A \defn{cover reflection} of $w$ is an inversion $t\in\inv(w)$ such that there exists some $s\in S$ with $tw=ws$.  The name comes from the fact that multiplying some element by a simple reflection produces a cover relation in the weak order, and the cover reflection is then the conjugate of this simple reflection by the larger element in this cover.  The set of cover reflections of $w$ is denoted by $\cov(w)$.

Now fix a reduced expression $\mathbf{w}=a_{1}a_{2}\cdots a_{k}$.  The \defn{inversion sequence} of $\mathbf{w}$ is the sequence $r_{1},r_{2},\ldots,r_{k}$, where $r_{i}=a_{1}a_{2}\cdots a_{i-1}a_{i}a_{i-1}\cdots a_{2}a_{1}$.  It is clear by construction that $\inv(w)=\{r_{1},r_{2},\ldots,r_{k}\}$.  Observe, however, that the inversion sequence equips $\inv(w)$ with a linear order; the \defn{inversion order} $r_{1}<r_{2}<\cdots<r_{k}$.  This order will be denoted by $\Inv{w}$.

Since the elements of $T$ geometrically act as reflections on a Euclidean vector space, we can associate two normal vectors to the corresponding reflecting hyperplane.  The collection of all these normal vectors is a \defn{root system} of $W$, and it can be partitioned into \defn{positive} and \defn{negative} roots.  It follows that there is a bijection from $T$ to the set $\Phi^{+}$ of all positive roots.  Given $\alpha\in\Phi^{+}$, let $t_{\alpha}\in T$ be the corresponding reflection.  It follows from \cite{dyer19on}*{Lemma~4.1(iv)} that whenever we have a reflection $t_{a\alpha+b\beta}\in\inv(w)$ for some $\alpha,\beta\in\Phi^{+}$ and some positive integers $a,b$, then at least one of $t_{\alpha}$ and $t_{\beta}$ are in $\inv(w)$ as well.  We refer the interested reader to \cite{humphreys90reflection} for more background on the geometric realization of Coxeter groups.

Finally, we recall the existence of some special, well-behaved reduced expressions for any element of $W$.  A \defn{Coxeter element} of $(W,S)$ is an element that has a reduced expression which is a permutation of all simple reflections.  Fix such a Coxeter element $c\in W$.  Clearly, any reduced expression of $w\in W$ appears as a subword of the half-infinite word $c^{\infty}$ (which is the infinite concatenation of a fixed reduced expression for $c$).  The \defn{$c$-sorting word} of $w$ is the reduced expression for $w$ which appears leftmost in $c^{\infty}$, and will be denoted by $\mathbf{w(c)}$.

\subsection{Aligned Elements for Parabolic Quotients}
  \label{sec:aligned_elements}
In \cite{reading07clusters}*{Section~4}, N.~Reading defined a notion of $c$-alignment for the elements of a finite Coxeter group $W$ with respect to some Coxeter element $c\in W$.  More precisely, an element $w$ of $W$ is \defn{$c$-aligned} if whenever we have $t_{\alpha}<t_{a\alpha+b\beta}<t_{\beta}$ in the inversion order $\Inv{\wo(c)}$, where $\alpha,\beta\in\Phi^{+}$ and $a,b$ are positive integers, then $t_{a\alpha+b\beta}\in\inv(w)$ implies $t_{\alpha}\in\inv(w)$.  For $W=\mathfrak{S}_{n}$ and $c=s_{n-1}\cdots s_{2}s_{1}$ the linear Coxeter element, the $c$-aligned elements are precisely the $231$-avoiding permutations.

We now propose a definition of $c$-aligned elements for parabolic quotients.

\begin{definition}\label{def:parabolic_alignment}
  Let $(W,S)$ be a finite Coxeter system, let $J\subseteq S$, and let $c\in W$ be a Coxeter element.  An element $w\in W^{J}$ is \defn{$(W^{J},c)$-aligned} if, whenever we have $t_{\alpha}<t_{a\alpha+b\beta}<t_{\beta}$ in $\Inv{\wo^{J}(c)}$, where $\alpha,\beta\in\Phi^{+}$ and $a,b$ are positive integers, then $t_{a\alpha+b\beta}\in\cov(w)$ implies $t_{\alpha}\in\inv(w)$.
\end{definition}

We denote the set of all $(W^{J},c)$-aligned elements of $W$ by $\Align(W^{J},c)$.

There is a subtlety in \Cref{def:parabolic_alignment}---we only require the root $t_{a\alpha+b\beta}$ to correspond to a cover reflection of $w$, rather than to an arbitrary inversion.  It was shown in \cite{reading06cambrian}*{Lemma~5.5} and follows from \cite{reading07clusters}*{Lemmas~4.9 and~4.11} that our parabolic aligned condition for $J=\emptyset$ is indeed equivalent to the original aligned condition for Coxeter groups of type $A, B,$ and $D$.  This equivalence is trivial for the dihedral groups, and it was checked by computer for the groups $H_{3},H_{4}$, and $F_{4}$.  The remaining exceptional groups $E_{6},E_{7}$, and $E_{8}$ have not been checked by computer.  See also \cite{williams13cataland}*{Remark~5.1.8}.

The following lemma states that our parabolic pattern avoidance condition from \Cref{def:parabolic_231} is equivalent to \Cref{def:parabolic_alignment} in the case of the symmetric group and the \defn{linear} Coxeter element $c=s_{1}s_{2}\cdots s_{n-1}$, where $s_{i}=(i,i+1)$.

\begin{lemma}\label{lem:type_a_alignment}
  Let $W=\mathfrak{S}_{n}$, $c=s_{1}s_{2}\cdots s_{n-1}$, and choose $J \subseteq S$.  An element $w\in\mathfrak{S}_{n}^{J}$ is $(W^{J},c)$-aligned if and only if $w^{-1}$ is ($J$,$231$)-avoiding.
\end{lemma}
\begin{proof}
  By definition, cover reflections of $w\in\mathfrak{S}_{n}$ correspond to descents of $w^{-1}$ so that \Cref{def:parabolic_alignment} agrees with \Cref{def:parabolic_compressed} after taking inverses.  \Cref{lem:j_compressed_avoiding} then implies that $w\in\mathfrak{S}_{n}^{J}$ is $(W^{J},c)$-aligned if and only if $w^{-1}$ is ($J$,$231$)-avoiding.
\end{proof}

It is an intriguing question whether the statement of \Cref{thm:parabolic_tamari_lattice} survives this generalization, and computer experiments have led us to formulate the following conjecture.

\begin{conjecture}\label{conj:parabolic_cambrian_lattice}
  For any finite Coxeter system $(W,S)$, any $J\subseteq S$, and any Coxeter element $c\in W$, the poset $\Weak\bigl(\Align(W^{J},c)\bigr)$ is a lattice.  Moreover, it is a lattice quotient of $\Weak(W^{J})$.
\end{conjecture}

\Cref{conj:parabolic_cambrian_lattice} holds in two interesting cases---when $J=\emptyset$ \cite{reading07sortable}*{Theorem~1.1}, and when $J=S\setminus\{s\}$ is chosen such that $\wo^{J}$ is \defn{fully commutative}, \ie any two reduced expressions for $\wo^{J}$ differ only by commutations.  In the latter case, we simply have $\Align\bigl(W^{J},c\bigr)=W^{J}$ for any $c$, and the lattice property follows from the fact that $(W^{J},\leq_{S})$ is an interval in the weak order on $W$~\cite{williams13cataland}*{Section~5.2}.  See \cite{stembridge96on} for more background on fully commutative elements, and a characterization of the sets $J=S\setminus\{s\}$ such that $\wo^{J}$ is fully commutative.  The remaining cases, however, are wide open.

\subsection{Noncrossing Partitions for Parabolic Quotients}
  \label{sec:noncrossing_partitions}
Let us continue with the generalization of noncrossing set partitions to parabolic quotients of finite Coxeter groups.  Recall that the original definition of noncrossing partitions associated with a pair $(W,c)$ is in terms of elements below $c$ in a certain partial order depending on \emph{all} reflections of $W$~\cites{bessis03dual,brady08non}.  It was observed by N.~Reading in \cite{reading07clusters}*{Theorem~6.1} that the $(W,c)$-noncrossing partitions are determined bijectively by the cover reflections of the $c$-aligned elements of $W$.  The next definition is a straight\-forward generalization of this correspondence.  For $w\in W^{J}$, and suppose that $\cov(w)=\{r_{1},r_{2},\ldots,r_{k}\}$, where the order of the cover reflections comes from $\Inv{w_{o}^{J}(c)}$.  Let $\psi(w)=r_{1}r_{2}\cdots r_{k}$.

\begin{definition}\label{def:parabolic_noncrossing}
  Let $(W,S)$ be a finite Coxeter system, let $J\subseteq S$, and let $c\in W$ be a Coxeter element.  The \defn{$(W^{J},c)$-noncrossing partitions} are the elements in the image of $\psi$ restricted to $\Align\bigl(W^{J},c\bigr)$.
\end{definition}

We denote the set of $(W^{J},c)$-noncrossing partitions by $\NC\bigl(W^{J},c\bigr)$.  It follows from \Cref{thm:j_sortables_noncrossings} that this definition coincides with \Cref{def:parabolic_noncrossing_set_partitions} when we consider the symmetric group and the linear Coxeter element.

\subsection{Nonnesting Partitions for Parabolic Quotients}
  \label{sec:nonnesting_partitions}
We now generalize the non\-nesting set partitions to parabolic quotients of finite Coxeter groups.  In the classical setting, nonnesting partitions are defined as follows for any finite irreducible Coxeter group that is not isomorphic to $H_{4}$.  If $W$ is an irreducible crystallographic Coxeter group, \ie we have $m_{ij}\in\{1,2,3,4,6\}$ in \eqref{eq:coxeter_presentation}, then we can partially order the positive roots of $W$ by $\alpha\leq\beta$ if and only if $\beta-\alpha$ can be expressed as a linear combination of the simple roots with only positive coefficients.  This partial order yields the \defn{root poset} of $W$.  Root posets for the remaining finite irreducible Coxeter groups other than $H_{4}$ were suggested in \cite{armstrong09generalized}*{Figure~5.15}.  The $W$-nonnesting partitions of $W$ are then the order ideals in the root poset.  In particular, they do not depend on a Coxeter element.  Recall that the \defn{parabolic root poset} of $W$ with respect to $J$ is defined to be the order filter in the root poset of $W$ induced by the simple reflections not in $J$.

\begin{definition}\label{def:parabolic_nonnesting}
  Let $(W,S)$ be a finite Coxeter system, with $W\neq H_{4}$, and let $J\subseteq S$.  The \defn{$W^{J}$-nonnesting partitions} are the order ideals in the parabolic root poset of $W$ with respect to $J$.
\end{definition}

We denote the set of $W^{J}$-nonnesting partitions by $\NN\bigl(W^{J}\bigr)$.  It is clear that this definition coincides with \Cref{def:parabolic_nonnesting_set_partitions} when we consider the symmetric group.

\subsection{Subword Complexes for Parabolic Quotients}
  \label{sec:subword_complexes}
However, this is not the end of the story.  There is yet another family of combinatorial objects that seems to fit nicely into the presented framework.  Let $(W,S)$ be a finite Coxeter system, let $Q$ be a word on the alphabet $S$, and let $w\in W$.  The \defn{subword complex} $\SW(Q,w)$ is the pure simplicial complex whose facets are the subwords $Q-P$ such that $P$ is a reduced expression for $w$~\cite{knutson04subword}.  For our purpose, the following subword complex shall be emphasized.

\begin{definition}
  Let $(W,S)$ be a finite Coxeter system, let $J\subseteq S$, and let $c\in W$ be a Coxeter element.  The \defn{$(W^{J},c)$-cluster complex} is the subword complex $\SW\bigl(\mathbf{c\wo(c)},\wo^{J}\bigr)$.
\end{definition}

We denote the $(W^{J},c)$-cluster complex by $\SW(W^{J},c)$, and denote its number of facets by $\bigl\lvert\SW(W^J,c)\bigr\rvert$.  The next result states that $\bigl\lvert\SW(\mathfrak{S}_{n}^J,c)\bigr\rvert$, where $c$ is the linear Coxeter element, equals the number appearing (implicitly) in \Cref{thm:parabolic_bijections}.

\begin{proposition}\label{prop:nonnesting_subword}
  Let $n>0$, let $S$ be the set of adjacent transpositions, and let $c$ be the linear Coxeter element.  For $J\subseteq S$, we have $\bigl\lvert\SW(\mathfrak{S}^{J},c)\bigr\rvert=\bigl\lvert\NN_{n}^{J}\bigr\rvert$.
\end{proposition}
\begin{proof}
  Recall from \Cref{sec:nonnesting} that the elements in $\NN_{n}^{J}$ are order ideals in the parabolic root poset of $\mathfrak{S}_{n}^{J}$, and this poset can be interpreted as the Ferrers shape $\lambda_n^J$.

  For any Ferrers shape $\lambda=(\lambda_{1},\lambda_{2},\ldots,\lambda_{t})$ we may consider the permutation
  \begin{displaymath}
    w(\lambda) := \prod_{i=1}^{t}\prod_{j=\lambda_{1}+2-i-\lambda_{i}}^{\lambda_{1}+1-i}{s_{j}}.
  \end{displaymath}

  It is straightforward to verify that in the case where $\lambda=\lambda_n^J$ describes the shape of the parabolic root poset of $\mathfrak{S}_{n}^{J}$ the element $w(\lambda)$ is precisely $\wo^{J}$.  The result follows then from \cite{serrano12maximal}*{Theorem~1.1}.  See also \cite{williams13cataland}*{Remark~4.5.8}.
\end{proof}

Since $\SW(W^{J},c)$ is a subword complex, there is a natural poset structure on its facets.  More generally, let $F,F'$ be two facets of a subword complex $\SW(Q,w)$ such that $F-\{i\}=F'-\{j\}$ for some $i\in F$, and some $j\in F'$.  If $i<j$, then we call $F\to F'$ a \defn{flip}, and the facets of $\SW(Q,w)$ together with the set of flips forms an acyclic graph, and therefore its reflexive and transitive closure is the \defn{flip poset} of $\SW(Q,w)$~\cite{knutson04subword}*{Remark~4.5}.

\begin{conjecture}\label{conj:flip_poset_aligned_poset}
  For any finite Coxeter system $(W,S)$, any $J\subseteq S$, and any Coxeter element $c\in W$, the restriction of the weak order to $(W^{J},c)$-aligned elements is isomorphic to the flip poset of $\SW(W^{J},c)$.
\end{conjecture}

\subsection{Numerology}
  \label{sec:numerology}
In this section we describe how the objects defined in \Cref{sec:aligned_elements,sec:noncrossing_partitions,sec:nonnesting_partitions,sec:subword_complexes} conjecturally fit together from an enumerative point of view.  It is well known that for any finite Coxeter group $W$ and any Coxeter element $c\in W$ we have
\begin{displaymath}
  \Bigl\lvert\Align\bigl(W^{\emptyset},c\bigr)\Bigr\rvert = \Bigl\lvert\NC\bigl(W^{\emptyset},c\bigr)\Bigr\rvert = \Bigl\lvert\NN\bigl(W^{\emptyset}\bigr)\Bigr\rvert = \Bigl\lvert\SW(W^{\emptyset},c)\Bigr\rvert,
\end{displaymath}
see for instance \cite{reading07sortable}*{Theorem~6.1} and \cite{armstrong13uniform}, and this cardinality is given by the well-known $W$-Catalan number~\cite{reiner97non}*{Remark~2}.  We have shown in \Cref{thm:parabolic_bijections} and \Cref{prop:nonnesting_subword} that this statement can be generalized to parabolic quotients of the symmetric group and the linear Coxeter element.  It turns out, however, that this statement does not hold in general for any parabolic quotient of any Coxeter group and any Coxeter element.  Take for instance $W=D_{4}$, $J=\{s_{1},s_{2}\}$, and $c=s_{3}s_{2}s_{1}s_{4}$, where $s_{2}$ is the simple reflection that does not commute with any of the other simple reflections.  In this case we have $\bigl\lvert\Align\bigl(W^{J},c\bigr)\bigr\rvert=21$, but $\bigl\lvert\NN\bigl(W^{J}\bigr)\bigr\rvert=22$.

Another related question is, whether the cardinality of the sets $\Align\bigl(W^{J},c\bigr)$ and $\NC\bigl(W^{J},c\bigr)$ is independent of the choice of $c$.  This property is known for $J=\emptyset$, see for instance \cite{reading07clusters}*{Theorem~9.1}, but it turns out once more that it does not hold for any parabolic quotient of any Coxeter group and any Coxeter element.  Take again $W=D_{4}$, $J=\{s_{1},s_{2}\}$, $c=s_{3}s_{2}s_{1}s_{4}$ as above and take $c'=s_{2}s_{3}s_{4}s_{1}$.  We then have $\bigl\lvert\Align\bigl(W^{J},c)\bigr\rvert=21$ and $\bigl\lvert\Align\bigl(W^{J},c')\bigr\rvert=22$.

As a consequence we conclude that, in general, there is no well-defined parabolic Coxeter-Catalan number $\Cat(W^{J})$.  See \Cref{fig:a4data,fig:d4data,fig:h3data,fig:h4data,fig:f4data} for more data.

\begin{table}[htbp]
  \centering
  \scalebox{0.9}{
  \begin{tabular}{|c|rcl|rcl|}
    \hline
    & $\bigl\lvert\Align(A_4^J)\bigr\rvert$ & $=$ & $\bigl\lvert\NC(A_4^J)\bigr\rvert$ & $\bigl\lvert\Align(B_4^J)\bigr\rvert$ & $=$ & $\bigl\lvert\NC(B_4^J)\bigr\rvert$ \\
    $J$ & & $=$ & $\bigl\lvert\SW(A_{4}^{J})\bigr\rvert$ & & $=$ & $\bigl\lvert\SW(B_{4}^{J})\bigr\rvert$ \\
    & & $=$ & $\bigl\lvert\NN(A_4^J)\bigr\rvert$ & & $=$ & $\bigl\lvert\NN(B_4^J)\bigr\rvert$ \\
    \hline
    $\{\}$ & \multicolumn{3}{|c|}{42} & \multicolumn{3}{|c|}{70}\\
    $\{s_1\}$ & \multicolumn{3}{|c|}{28} & \multicolumn{3}{|c|}{50}\\
    $\{s_2\}$ & \multicolumn{3}{|c|}{32} & \multicolumn{3}{|c|}{58}\\
    $\{s_3\}$ & \multicolumn{3}{|c|}{32} & \multicolumn{3}{|c|}{60}\\
    $\{s_4\}$ & \multicolumn{3}{|c|}{28} & \multicolumn{3}{|c|}{56}\\
    $\{s_1, s_2\}$ & \multicolumn{3}{|c|}{14} & \multicolumn{3}{|c|}{30}\\
    $\{s_1, s_3\}$ & \multicolumn{3}{|c|}{22} & \multicolumn{3}{|c|}{44}\\
    $\{s_1, s_4\}$ & \multicolumn{3}{|c|}{19} & \multicolumn{3}{|c|}{41}\\
    $\{s_2, s_3\}$ & \multicolumn{3}{|c|}{17} & \multicolumn{3}{|c|}{40}\\
    $\{s_2, s_4\}$ & \multicolumn{3}{|c|}{22} & \multicolumn{3}{|c|}{48}\\
    $\{s_3, s_4\}$ & \multicolumn{3}{|c|}{14} & \multicolumn{3}{|c|}{28}\\
    $\{s_1, s_2, s_3\}$ & \multicolumn{3}{|c|}{5} & \multicolumn{3}{|c|}{16}\\
    $\{s_1, s_2, s_4\}$ & \multicolumn{3}{|c|}{10} & \multicolumn{3}{|c|}{26}\\
    $\{s_1, s_3, s_4\}$ & \multicolumn{3}{|c|}{10} & \multicolumn{3}{|c|}{22}\\
    $\{s_2, s_3, s_4\}$ & \multicolumn{3}{|c|}{5} & \multicolumn{3}{|c|}{8}\\
    $\{s_1, s_2, s_3, s_4\}$ & \multicolumn{3}{|c|}{1} & \multicolumn{3}{|c|}{1}\\
    \hline
  \end{tabular}}
  \caption{The numbers $\Cat(W^J)$ for $W\in\{A_4,B_4\}$.  For $B_{4}$ the noncommuting simple reflections satisfy $(s_{1}s_{2})^{3}=(s_{2}s_{3})^{3}=(s_{3}s_{4})^{4}=\id$.  We have suppressed the dependence of $\bigl\lvert\Align(W^J,c)\bigr\rvert,\bigl\lvert\NC(W^J,c)\bigl\rvert$ and $\bigl\lvert\SW(W^{J},c)\bigr\rvert$ on $c$, as they agree for all Coxeter elements.}
  \label{fig:a4data}
\end{table}

\begin{table}[htbp]
  \centering
  \begin{tabular}{|c|cccc|c|}
    \hline
    $J$ & \multicolumn{4}{c|}{$\bigl\lvert\Align(D_4^J,c)\bigr\rvert=\bigl\lvert\NC(D_4^J,c)\bigr\rvert=\bigl\lvert\SW(D_{4}^{J},c)\bigr\rvert$} & $\bigl\lvert\NN(D_4^J)\bigr\rvert$ \\
    \hline
    & $s_2s_3s_4s_1$ & $s_1s_2s_3s_4$ & $s_3s_2s_1s_4$ & $s_4s_2s_3s_1$  & \\ \hline
    $\{\}$ & \multicolumn{4}{c|}{50} & 50\\
    $\{s_1\}$ & \multicolumn{4}{c|}{36} & 36\\
    $\{s_2\}$ & \multicolumn{4}{c|}{42} & 42\\
    $\{s_3\}$ & \multicolumn{4}{c|}{36} & 36\\
    $\{s_4\}$ & \multicolumn{4}{c|}{36} & 36\\
    \rowcolor[gray]{.8} $\{s_1, s_2\}$ & 22 & 22 & 21 & 21 & 22\\
    $\{s_1, s_3\}$ & \multicolumn{4}{c|}{27} & 27\\
    $\{s_1, s_4\}$ &  \multicolumn{4}{c|}{27} & 27\\
    \rowcolor[gray]{.8}$\{s_2, s_3\}$ & 22 & 21 & 22 & 21 & 22\\
    \rowcolor[gray]{.8}$\{s_2, s_4\}$ & 22 & 21 & 21 & 22 & 22\\
    $\{s_3, s_4\}$ & \multicolumn{4}{c|}{27} & 27\\
    $\{s_1, s_2, s_3\}$ & \multicolumn{4}{c|}{8} & 8\\
    $\{s_1, s_2, s_4\}$ & \multicolumn{4}{c|}{8} & 8\\
    $\{s_1, s_3, s_4\}$ & \multicolumn{4}{c|}{21} & 21\\
    $\{s_2, s_3, s_4\}$ & \multicolumn{4}{c|}{8} & 8\\
    $\{s_1, s_2, s_3, s_4\}$ & \multicolumn{4}{c|}{1} & 1\\
    \hline
  \end{tabular}
  \caption{The various numbers $\Cat(D_4^J)$.  Here $s_{2}$ is the unique simple reflection that does not commute with the other simple reflections.  The values of $\bigl\lvert\Align(D_{4}^J,c)\bigr\rvert=\bigl\lvert\NC(D_{4}^J,c)\bigr\lvert=\bigl\lvert\SW(D_{4}^{J},c)\bigr\rvert$ are equal for $c$ and $c^{-1}$.}
  \label{fig:d4data}
\end{table}

\begin{table}[htbp]
  \centering
  \begin{tabular}{|c|c|} \hline
    $J$ & $\bigl\lvert\Align(H_3^J)\bigr\rvert=\bigl\lvert\NC(H_3^J)\bigr\rvert=\bigl\lvert\SW(H_{3}^{J})\bigr\rvert=\bigl\lvert\NN(H_3^J)\bigl\rvert$ \\
    \hline
    $\{\}$ & 32\\
    $\{s_1\}$ & 27\\
    $\{s_2\}$ & 28\\
    $\{s_3\}$ & 25\\
    $\{s_1, s_2\}$ & 12\\
    $\{s_1, s_3\}$ & 22\\
    $\{s_2, s_3\}$ & 18\\
    $\{s_1, s_2, s_3\}$ & 1\\
    \hline
  \end{tabular}
  \caption{The numbers $\Cat(H_3^J)$, where the noncommuting simple reflections satisfy $(s_{1}s_{2})^{5}=(s_{2}s_{3})^{3}=\id$.  As in \Cref{fig:a4data}, we have suppressed the dependence of $\bigl\lvert\Align(H_{3}^J,c)\bigr\rvert,\bigl\lvert\NC(H_{3}^J,c)\bigr\rvert$ and $\bigl\lvert\SW(H_{3}^{J},c)\bigr\rvert$ on $c$.}
  \label{fig:h3data}
\end{table}

\begin{table}[htbp]
  \centering
  \begin{tabular}{|c|c|c|}
    \hline
    $J$ & $\bigl\lvert\Align(H_4^J)\bigr\rvert=\bigl\lvert\NC(H_4^J)\bigr\rvert=\bigl\lvert\SW(H_{4}^{J})\bigr\rvert$ & $\bigl\lvert\NN(H_4^J)\bigr\rvert$ \\ \hline
    $\{\}$ & 280 & 280\\
    $\{s_1\}$ & 266 & 266\\
    $\{s_2\}$ & 270 & 270\\
    $\{s_3\}$ & 266 & 266\\
    $\{s_4\}$ & 248 & 248\\
    \rowcolor[gray]{.8}$\{s_1, s_2\}$ & 209 & 210\\
    $\{s_1, s_3\}$ & 256 & 256\\
    $\{s_1, s_4\}$ & 239 & 239\\
    $\{s_2, s_3\}$ & 245 & 245\\
    $\{s_2, s_4\}$ & 242 & 242\\
    $\{s_3, s_4\}$ & 216 & 216\\
    \rowcolor[gray]{.8}$\{s_1, s_2, s_3\}$ & 95 & 106\\
    \rowcolor[gray]{.8}$\{s_1, s_2, s_4\}$ & 197 & 198\\
    $\{s_1, s_3, s_4\}$ & 212 & 212\\
    $\{s_2, s_3, s_4\}$ & 191 & 191\\
    $\{s_1, s_2, s_3, s_4\}$ & 1 & 1\\
    \hline
  \end{tabular}
  \caption{The various numbers $\Cat(H_4^J)$, where the noncommuting simple reflections satisfy $(s_{1}s_{2})^{5}=(s_{2}s_{3})^{3}=(s_{3}s_{4})^{3}=\id$.  As in \Cref{fig:a4data}, we have suppressed the dependence of $\bigl\lvert\Align(H_{4}^J,c)\bigr\rvert,\bigl\lvert\NC(H_{4}^J,c)\bigr\rvert$ and $\bigl\lvert\SW(H_{4}^{J},c)\bigr\rvert$ on $c$.  The values for $\bigl\lvert\NN(H_{4}^J)\bigr\rvert$ were computed using the four candidate ``root posets'' in Figure 5 of~\cite{cuntz15root}, all of which gave the same numbers.}
  \label{fig:h4data}
\end{table}

\begin{table}[htbp]
  \centering
  \begin{tabular}{|c|cccc|c|} \hline
    $J$ & \multicolumn{4}{c|}{$\bigl\lvert\Align(F_4^J,c)\bigr\rvert=\bigl\lvert\NC(F_4^J,c)\bigr\rvert=\bigl\lvert\SW(F_{4}^{J},c)\bigr\rvert$} & $|\NN(F_4^J)|$ \\
    \hline
    & $s_2s_3s_4s_1$ & $s_1s_2s_3s_4$ & $s_3s_2s_1s_4$ & $s_4s_2s_3s_1$ & \\
    \hline
    $\{\}$ & \multicolumn{4}{c|}{105} & 105\\
    $\{s_1\}$ &\multicolumn{4}{c|}{85} & 85\\
    $\{s_2\}$ & \multicolumn{4}{c|}{95} & 95\\
    $\{s_3\}$ & \multicolumn{4}{c|}{95} & 95\\
    $\{s_4\}$ &\multicolumn{4}{c|}{85} & 85\\
    $\{s_1, s_2\}$ & \multicolumn{4}{c|}{65} & 65\\
    $\{s_1, s_3\}$ & \multicolumn{4}{c|}{79} & 79\\
    $\{s_1, s_4\}$ &  \multicolumn{4}{c|}{71} & 71 \\
    \rowcolor[gray]{.8}$\{s_2, s_3\}$ & 62 & 57 & 62 & 62 & 63 \\
    $\{s_2, s_4\}$ & \multicolumn{4}{c|}{79} & 79\\
    $\{s_3, s_4\}$ & \multicolumn{4}{c|}{65} & 65\\
    \rowcolor[gray]{.8} $\{s_1, s_2, s_3\}$ & 23 & 23 & 23 & 23 & 24\\
    $\{s_1, s_2, s_4\}$ & \multicolumn{4}{c|}{57} & 57\\
    $\{s_1, s_3, s_4\}$ & \multicolumn{4}{c|}{57} & 57\\
    $\{s_2, s_3, s_4\}$ & \multicolumn{4}{c|}{23} & 23\\
    $\{s_1, s_2, s_3, s_4\}$ & \multicolumn{4}{c|}{1} & 1 \\
    \hline
  \end{tabular}
  \caption{The various numbers $\Cat(F_4^J)$.  Here we have the relations $(s_{1}s_{2})^{3}=(s_{2}s_{3})^{4}=(s_{3}s_{4})^{3}=\id$ between noncommuting simple reflections.  The values of $\bigl\lvert\Align(W^J,c)\bigr\rvert=\bigl\lvert\NC(W^J,c)\bigr\rvert=\bigl\lvert\SW(F_{4}^{J},c)\bigr\rvert$ are equal for $c$ and $c^{-1}$.}
  \label{fig:f4data}
\end{table}

In studying these tables, we observe that for the groups $A_{4},B_{4},H_{3}$ (and trivially for the dihedral groups) there seem to exist well-defined parabolic Coxeter-Catalan numbers.  Further computer experiments suggest the following conjecture.

\begin{conjecture}\label{conj:parabolic_catalan_exceptional}
  Let $(W,S)$ be a Coxeter system with $W\in\bigl\{A_{n},B_{n},H_{3},I_{2}(m)\bigr\}$, and let $J\subseteq S$.  For any Coxeter element $c\in W$ the cardinalities of the sets $\Align\bigl(W^{J},c\bigr)$, $\NC\bigl(W^{J},c\bigr)$, $\SW\bigl(W^{J},c\bigr)$, and $\NN\bigl(W^{J}\bigr)$ are equal, and hence do not depend on the choice of $c$.
\end{conjecture}

The groups appearing in \Cref{conj:parabolic_catalan_exceptional} are sometimes referred to as the ``co\-incidental types'', because they share remarkable features that distinguish them from the other finite Coxeter groups.  Some of these features can be found in \cite{fomin05generalized}*{Theorems~8.5~and~10.2}, \cite{miller15foulkes}*{Theorem~14}, \cite{muehle15sb}*{Theorem~2}, \cite{reading08chains}, and \cite{williams13cataland}*{Remark~3.1.26}.  Since for these groups the families of parabolic aligned elements, parabolic noncrossing and nonnesting partitions are equinumerous, we are tempted to define a parabolic Coxeter-Catalan number as follows.

\begin{definition}\label{def:parabolic_catalan_number}
  Let $(W,S)$ be a Coxeter system with $W\in\{A_{n},B_{n},H_{3},I_{2}(m)\}$, and let $J\subseteq S$.  Define the \defn{parabolic Coxeter-Catalan number} by
  \begin{displaymath}
    \Cat(W^{J})=\bigl\lvert\NN(W^{J})\bigr\rvert.
  \end{displaymath}
\end{definition}

\subsection{Aligned Elements for Arbitrary Reduced Expressions}
  \label{sec:aligned_expressions}
Observe that the definition of the $(W^{J},c)$-aligned elements from \Cref{def:parabolic_alignment} does not so much depend on the fact that we consider a parabolic quotient of a Coxeter group, rather than on the particular reduced expression of $\wo^{J}$ we have chosen.  More precisely, the alignment property depends on the inversion order $\Inv{\wo^{J}}$.  This suggests the following definition.

\begin{definition}\label{def:general_alignment}
  Let $(W,S)$ be a Coxeter system, let $w\in W$, and fix a reduced expression $\mathbf{w}$ for $w$.  An element $x\leq_{S}w$ is \defn{$\mathbf{w}$-aligned} if whenever we have $t_{\alpha}<t_{a\alpha+b\beta}<t_{\beta}$ in the inversion order $\Inv{w}$ for $a,b$ positive integers, then $t_{a\alpha+b\beta}\in\cov(x)$ implies $t_{\alpha}\in\inv(x)$.
\end{definition}

In particular, this definition requires that $t_{\alpha},t_{a\alpha+b\beta},t_{\beta}\in\inv(w)$.  Let $\Align(W,\mathbf{w})$ denote the set of all $\mathbf{w}$-aligned elements of $W$.  Note that at this level of generality we do not even need to require that $W$ is finite, and we can pick any element $w\in W$.  It is immediate that if $W$ is finite and $c\in W$ is a Coxeter element, then $x\in W$ is $c$-aligned if and only if it is $\mathbf{\wo(c)}$-aligned.  Let us illustrate \Cref{def:general_alignment} with an example.

\begin{figure}
  \centering
  \begin{tikzpicture}
    \def\x{1.75};
    \def\y{1.25};
    \draw(3.5*\x,1*\y) node[fill=gray!50!white](n1){$e$};
    \draw(2.5*\x,2*\y) node[fill=gray!50!white](n2){$s_{1}$};
    \draw(3.5*\x,2*\y) node[fill=gray!50!white](n3){$s_{0}$};
    \draw(4.5*\x,2*\y) node[fill=gray!50!white](n4){$s_{3}$};
    \draw(1.5*\x,3*\y) node(n5){$s_{1}s_{0}$};
    \draw(2.5*\x,3*\y) node[fill=gray!50!white](n6){$s_{0}s_{1}$};
    \draw(3.5*\x,3*\y) node[fill=gray!50!white](n7){$s_{1}s_{3}$};
    \draw(4.5*\x,3*\y) node[fill=gray!50!white](n8){$s_{0}s_{3}$};
    \draw(5.5*\x,3*\y) node(n9){$s_{3}s_{0}$};
    \draw(1*\x,4*\y) node[fill=gray!50!white](n10){$s_{1}s_{0}s_{3}$};
    \draw(2*\x,4*\y) node[fill=gray!50!white](n11){$s_{0}s_{1}s_{0}$};
    \draw(3*\x,4*\y) node[fill=gray!50!white](n12){$s_{0}s_{1}s_{3}$};
    \draw(4*\x,4*\y) node(n13){$s_{1}s_{3}s_{0}$};
    \draw(5*\x,4*\y) node(n14){$s_{3}s_{0}s_{1}$};
    \draw(6*\x,4*\y) node[fill=gray!50!white](n15){$s_{0}s_{3}s_{0}$};
    \draw(1.5*\x,5*\y) node[fill=gray!50!white](n16){$s_{0}s_{1}s_{0}s_{3}$};
    \draw(2.5*\x,5*\y) node(n17){$s_{1}s_{0}s_{3}s_{0}$};
    \draw(3.5*\x,5*\y) node(n18){$s_{0}s_{1}s_{3}s_{0}$};
    \draw(4.5*\x,5*\y) node(n19){$s_{1}s_{3}s_{0}s_{1}$};
    \draw(5.5*\x,5*\y) node(n20){$s_{0}s_{3}s_{0}s_{1}$};
    \draw(2.5*\x,6*\y) node[fill=gray!50!white](n21){$s_{0}s_{1}s_{0}s_{3}s_{0}$};
    \draw(3.5*\x,6*\y) node(n22){$s_{1}s_{0}s_{3}s_{0}s_{1}$};
    \draw(4.5*\x,6*\y) node[fill=gray!50!white](n23){$s_{0}s_{1}s_{3}s_{0}s_{1}$};
    \draw(3.5*\x,7*\y) node[fill=gray!50!white](n24){$s_{0}s_{1}s_{0}s_{3}s_{0}s_{1}$};
    \draw(5*\x,7*\y) node[fill=gray!50!white](n25){$s_{1}s_{0}s_{3}s_{0}s_{1}s_{2}$};
    \draw(5*\x,8*\y) node[fill=gray!50!white](n26){$s_{0}s_{1}s_{0}s_{3}s_{0}s_{1}s_{2}$};
    \draw(n1) -- (n2);
    \draw(n1) -- (n3);
    \draw(n1) -- (n4);
    \draw(n2) -- (n5);
    \draw(n2) -- (n7);
    \draw(n3) -- (n6);
    \draw(n3) -- (n8);
    \draw(n4) -- (n7);
    \draw(n4) -- (n9);
    \draw(n5) -- (n10);
    \draw(n5) -- (n11);
    \draw(n6) -- (n11);
    \draw(n6) -- (n12);
    \draw(n7) -- (n13);
    \draw(n8) -- (n12);
    \draw(n8) -- (n15);
    \draw(n9) -- (n14);
    \draw(n9) -- (n15);
    \draw(n10) -- (n16);
    \draw(n10) -- (n17);
    \draw(n11) -- (n16);
    \draw(n12) -- (n18);
    \draw(n13) -- (n17);
    \draw(n13) -- (n19);
    \draw(n14) -- (n19);
    \draw(n14) -- (n20);
    \draw(n15) -- (n20);
    \draw(n16) -- (n21);
    \draw(n17) -- (n22);
    \draw(n18) -- (n21);
    \draw(n18) -- (n23);
    \draw(n19) -- (n22);
    \draw(n20) -- (n23);
    \draw(n21) -- (n24);
    \draw(n22) -- (n24);
    \draw(n22) -- (n25);
    \draw(n23) -- (n24);
    \draw(n24) -- (n26);
    \draw(n25) -- (n26);
  \end{tikzpicture}
  \caption{The weak order interval $[e,w]$, where $w$ is given by the reduced expression $\mathbf{w}=s_{0}s_{1}s_{0}s_{3}s_{0}s_{1}s_{2}$ in the affine Coxeter group $\tilde{A}_{3}$.  The $\mathbf{w}$-aligned elements are highlighted in gray.}
  \label{fig:affine_a3_poset}
\end{figure}

\begin{example}\label{ex:affine_a3_aligned}
  Let $W=\tilde{A}_{3}$ be the affine symmetric group of rank $4$.  Denote its simple reflections by $s_{0},s_{1},s_{2},s_{3}$ such that the following Coxeter relations hold:
  \begin{displaymath}
    (s_{0}s_{1})^{3} = (s_{1}s_{2})^{3} = (s_{2}s_{3})^{3} = (s_{0}s_{3})^{3} = (s_{0}s_{2})^{2} = (s_{1}s_{3})^{2} = \id.
  \end{displaymath}
  Pick $w\in\tilde{A}_{3}$ given by the reduced expression $\mathbf{w}=s_{0}s_{1}s_{0}s_{3}s_{0}s_{1}s_{2}$.  The weak order interval $[\id,w]$ is shown in \Cref{fig:affine_a3_poset}.  The inversion order $\Inv{w}$ is given by
  \begin{displaymath}
    s_{0} < s_{0}s_{1}s_{0} < s_{1} < s_{1}s_{0}s_{3}s_{0}s_{1} < s_{0}s_{3}s_{0} < s_{3} < s_{1}s_{0}s_{3}s_{0}s_{1}s_{2}s_{1}s_{0}s_{3}s_{0}s_{1}.
  \end{displaymath}
  Let us denote these reflections by $t_{1},t_{2},t_{3},t_{4},t_{5},t_{6},t_{7}$ in that order; and let $\beta_{i}$ be the positive root corresponding to $t_{i}$ for $i\in[7]$.  The roots $\beta_{1},\beta_{3},\beta_{6}$ are simple; and we have the following decompositions:
  \begin{displaymath}\begin{aligned}
    & \beta_{2} = \beta_{1}+\beta_{3}, && \beta_{4} = \beta_{2}+\beta_{6} = \beta_{3}+\beta_{5}, && \beta_{5} = \beta_{1}+\beta_{6}.
  \end{aligned}\end{displaymath}
  The root $\beta_{7}$ is not simple, but cannot be written as a (nontrivial) linear combination of any of the $\beta_{i}$'s.  In view of \Cref{def:general_alignment} an element $x\leq_{S}w$ is $\mathbf{w}$-aligned if whenever it has $t_{2}$, or $t_{4}$, or $t_{5}$ as a cover reflection, then it needs to have $t_{1}$, or $t_{2}$ and $t_{3}$, or $t_{1}$, respectively, as inversions.  This is satisfied for the elements in \Cref{fig:affine_a3_poset} highlighted in gray.  If we consider $x=s_{1}s_{0}s_{3}s_{0}$, then we can check that $\cov(x)=\{t_{2},t_{6}\}$ and $\inv(x)=\{t_{2},t_{3},t_{4},t_{6}\}$.  Therefore, $x$ is not $\mathbf{w}$-aligned.
\end{example}

It is tempting to conjecture that the weak order on $\mathbf{w}$-aligned elements always forms a lattice (and therefore to extend \Cref{conj:parabolic_cambrian_lattice} to the more general setting of \Cref{def:general_alignment}).  However, this turns out to be false, even in finite type.  If we take $W=A_{4}$ and $w\in W$ given by the reduced expression $\mathbf{w}=s_{2}s_{1}s_{2}s_{3}s_{4}s_{2}s_{1}$, then there are twenty $\mathbf{w}$-aligned elements, but $\Weak\bigl(\Align(A_4,\mathbf{w})\bigr)$ is not a lattice.  So far, \Cref{ex:affine_a3_aligned} shows the smallest poset of $\mathbf{w}$-aligned elements known to us that is not a lattice under weak order.  (Note that the elements $s_{0}s_{1}s_{0}s_{3}s_{0}s_{1}$ and $s_{1}s_{0}s_{3}s_{0}s_{1}s_{2}$ have two maximal $\mathbf{w}$-aligned lower bounds, namely $s_{1}s_{0}s_{3}$ and $s_{1}s_{3}$.)

We have not been able to determine necessary and sufficient conditions on $W$ and $\mathbf{w}$ such that $\Weak\bigl(\Align(W,\mathbf{w})\bigr)$ is a lattice.  It turns out that the next best candidate, namely the conjecture that $\Weak\bigl(\Align(W,\mathbf{w(c)})\bigr)$ for some Coxeter element $c$ is always a lattice, is also wrong.  Consider again $W=A_{4}$ and $\mathbf{w}=s_{3}s_{4}s_{1}s_{3}s_{2}s_{1}s_{3}s_{4}$.  This is a $s_{3}s_{4}s_{2}s_{1}$-sorting word, but the corresponding weak order poset is not a lattice.

We are, however, not aware of any counterexamples in rank $3$.

\section*{Acknowledgements}

The second author would like to thank D.~Stanton, V.~Reiner, and H.~Thomas for their support and guidance, N.~Reading and C.~Ceballos for helpful conversations, and C.~Arreche for his proofreading.

\bibliography{literature_essential}

\begin{bibdiv}
\begin{biblist}

\bib{armstrong09generalized}{article}{
      author={Armstrong, D.},
       title={{Generalized Noncrossing Partitions and Combinatorics of Coxeter
  Groups}},
        date={2009},
     journal={{Mem. Amer. Math. Soc.}},
      volume={202},
}

\bib{armstrong13uniform}{article}{
      author={Armstrong, D.},
      author={Stump, C.},
      author={Thomas, H.},
       title={{A Uniform Bijection between Nonnesting and Noncrossing
  Partitions}},
        date={2013},
     journal={Trans. Amer. Math. Soc.},
      volume={365},
       pages={4121\ndash 4151},
}

\bib{athanasiadis98on}{article}{
      author={Athanasiadis, C.~A.},
       title={{On Noncrossing and Nonnesting Partitions for Classical
  Reflection Groups}},
        date={1998},
     journal={Elec. J. Combin.},
}

\bib{bessis03dual}{article}{
      author={Bessis, D.},
       title={{The Dual Braid Monoid}},
        date={2003},
     journal={Ann. Sci. {\'E}cole Norm. Sup.},
      volume={36},
       pages={647\ndash 683},
}

\bib{bjorner05combinatorics}{book}{
      author={Bj{\"o}rner, A.},
      author={Brenti, F.},
       title={{Combinatorics of Coxeter Groups}},
   publisher={Springer},
     address={New York},
        date={2005},
}

\bib{bjorner88generalized}{article}{
      author={Bj{\"o}rner, A.},
      author={Wachs, M.~L.},
       title={{Generalized Quotients in Coxeter Groups}},
        date={1988},
     journal={Trans. Amer. Math. Soc.},
      volume={308},
       pages={1\ndash 37},
}

\bib{bjorner97shellable}{article}{
      author={Bj{\"o}rner, A.},
      author={Wachs, M.~L.},
       title={{Shellable and Nonpure Complexes and Posets II}},
        date={1997},
     journal={Trans. Amer. Math. Soc.},
      volume={349},
       pages={3945\ndash 3975},
}

\bib{brady08non}{article}{
      author={Brady, T.},
      author={Watt, C.},
       title={{Non-Crossing Partition Lattices in Finite Real Reflection
  Groups}},
        date={2008},
     journal={Trans. Amer. Math. Soc.},
      volume={360},
       pages={1983\ndash 2005},
}

\bib{cuntz15root}{article}{
      author={Cuntz, M.},
      author={Stump, C.},
       title={{On Root Posets for Noncrystallographic Root Systems}},
        date={2015},
     journal={Math. Comp.},
      volume={84},
       pages={485\ndash 503},
}

\bib{dyer19on}{article}{
      author={Dyer, M.},
       title={{On the Weak Order of Coxeter Groups}},
        date={2019},
     journal={Can. J. Math.},
      volume={71},
       pages={299\ndash 336},
}

\bib{fomin05generalized}{article}{
      author={Fomin, S.},
      author={Reading, N.},
       title={{Generalized Cluster Complexes and Coxeter Combinatorics}},
        date={2005},
     journal={Int. Math. Res. Notes},
      volume={44},
       pages={2709\ndash 2757},
}

\bib{gratzer11lattice}{book}{
      author={Gr{\"a}tzer, G.},
       title={{Lattice Theory: Foundation}},
   publisher={Springer},
     address={Basel},
        date={2011},
}

\bib{humphreys90reflection}{book}{
      author={Humphreys, J.~E.},
       title={{Reflection Groups and Coxeter Groups}},
   publisher={Cambridge University Press},
     address={Cambridge},
        date={1990},
}

\bib{knutson04subword}{article}{
      author={Knutson, A.},
      author={Miller, E.},
       title={{Subword Complexes in Coxeter Groups}},
        date={2004},
     journal={Adv. Math.},
      volume={184},
       pages={161\ndash 176},
}

\bib{kreweras65sur}{article}{
      author={Kreweras, G.},
       title={Sur une classe de probl{\`e}mes de d{\'e}nombrement li{\'e}s au
  treillis des partitions des entiers},
        date={1965},
     journal={Cahiers du B.U.R.O. Univ. de Rech. Op{\'e}r.},
      volume={6},
       pages={5\ndash 105},
}

\bib{kreweras72sur}{article}{
      author={Kreweras, G.},
       title={Sur les partitions non crois\'ees d'un cycle},
        date={1972},
     journal={Discrete Math.},
      volume={1},
       pages={333\ndash 350},
}

\bib{miller15foulkes}{article}{
      author={Miller, A.~R.},
       title={{Foulkes Characters for Complex Reflection Groups}},
        date={2015},
     journal={Proc. Amer. Math. Soc.},
      volume={143},
       pages={3281\ndash 3293},
}

\bib{muehle15sb}{article}{
      author={M{\"u}hle, H.},
       title={{SB-Labelings, Distributivity, and Bruhat Order on Sortable
  Elements}},
        date={2015},
     journal={Electron. J. Combin.},
      volume={22},
}

\bib{muehle15tamari}{article}{
      author={M{\"u}hle, H.},
      author={Williams, N.},
       title={{Tamari Lattices for Parabolic Quotients of the Symmetric
  Group}},
        date={2015},
     journal={Discrete Math. Theor. Comput. Sci.},
       pages={973\ndash 984},
}

\bib{proctor17parabolic1}{article}{
      author={Proctor, R.~A.},
      author={Willis, M.~J.},
       title={{Parabolic Catalan Numbers Count Efficient Inputs for
  Gessel-Viennot Flagged Schur Function Determinant}},
        date={2017},
     journal={arXiv:1701.01182},
}

\bib{proctor17parabolic2}{article}{
      author={Proctor, R.~A.},
      author={Willis, M.~J.},
       title={{Parabolic Catalan Numbers Count Flagged Schur Functions and
  their Appearances as Type $A$ Demazure Characters (Key Polynomials)}},
        date={2017},
     journal={Discrete Math. Theor. Comput. Sci.},
      volume={19},
}

\bib{proctor18convexity}{article}{
      author={Proctor, R.~A.},
      author={Willis, M.~J.},
       title={{Convexity of Tableau Sets for Type $A$ Demazure Characters (Key
  Polynomials), Parabolic Catalan Numbers}},
        date={2018},
     journal={Discrete Math. Theor. Comput. Sci.},
      volume={20},
}

\bib{reading06cambrian}{article}{
      author={Reading, N.},
       title={{Cambrian Lattices}},
        date={2006},
     journal={Adv. Math.},
      volume={205},
       pages={313\ndash 353},
}

\bib{reading07clusters}{article}{
      author={Reading, N.},
       title={{Clusters, Coxeter-Sortable Elements and Noncrossing
  Partitions}},
        date={2007},
     journal={Trans. Amer. Math. Soc.},
      volume={359},
       pages={5931\ndash 5958},
}

\bib{reading07sortable}{article}{
      author={Reading, N.},
       title={{Sortable Elements and Cambrian Lattices}},
        date={2007},
     journal={Algebra Universalis},
      volume={56},
       pages={411\ndash 437},
}

\bib{reading08chains}{article}{
      author={Reading, N.},
       title={{Chains in the Noncrossing Partition Lattice}},
        date={2008},
     journal={SIAM J. Discrete Math.},
      volume={22},
       pages={875\ndash 886},
}

\bib{reading15noncrossing}{article}{
      author={Reading, N.},
       title={{Noncrossing Arc Diagrams and Canonical Join Representations}},
        date={2015},
     journal={SIAM J. Discrete Math.},
      volume={29},
       pages={736\ndash 750},
}

\bib{reiner97non}{article}{
      author={Reiner, V.},
       title={{Non-Crossing Partitions for Classical Reflection Groups}},
        date={1997},
     journal={Discrete Math.},
      volume={177},
       pages={195\ndash 222},
}

\bib{serrano12maximal}{article}{
      author={Serrano, L.},
      author={Stump, C.},
       title={{Maximal Fillings of Moon Polyominoes, Simplicial Complexes, and
  Schubert Polynomials}},
        date={2012},
     journal={Electron. J. Combin.},
      volume={19},
}

\bib{stembridge96on}{article}{
      author={Stembridge, J.~R.},
       title={{On the Fully Commutative Elements of Coxeter Groups}},
        date={1996},
     journal={J. Algebraic Combin.},
      volume={5},
       pages={353\ndash 385},
}

\bib{tamari51monoides}{thesis}{
      author={Tamari, D.},
       title={{Monoides Pr{\'e}ordonn{\'e}s et Cha{\^i}nes de Malcev}},
        type={Th{\`e}se de Math{\'e}matiques},
 institution={Universit{\'e} de Paris},
        date={1951},
}

\bib{williams13cataland}{thesis}{
      author={Williams, N.},
       title={{Cataland}},
        type={Dissertation},
 institution={University of Minnesota},
        date={2013},
}

\end{biblist}
\end{bibdiv}

\end{document}